\documentclass[letterpaper,12pt]{amsart}
\usepackage[margin=1in]{geometry}
\usepackage[dvipsnames]{xcolor}
\usepackage{ytableau}
\usepackage{verbatim}
\usepackage{color}
\usepackage{amsmath, amsthm}
\usepackage{amssymb}
\usepackage{todonotes}
\usepackage{tikz}
\usetikzlibrary{cd}
\usepackage{bbm}
\usepackage{adjustbox}
\usepackage{multirow}
\usepackage{multicol}
\usepackage{float, comment}
\usetikzlibrary{decorations.markings}

\usepackage[pdftex,colorlinks,citecolor=blue,linkcolor=red]{hyperref}

\theoremstyle{plain}
\newtheorem{theorem}{Theorem}[section]
\newtheorem{lemma}[theorem]{Lemma}

\newtheorem{corollary}[theorem]{Corollary}

\theoremstyle{definition}
\newtheorem{definition}[theorem]{Definition}
\newtheorem{example}[theorem]{Example}
\newtheorem{remark}[theorem]{Remark}

\def\phi{{\varphi}}
\def\cM{{\mathcal{M}}}

\newcommand{\cO}{\mathcal{O}}

\newcommand{\lcm}[1]{\mathrm{lcm}({#1})}

\def\T{\ytableausetup{boxsize=0.6cm}
\begin{ytableau}
    1 & 2 & 6 & 7 & 14 & 19 \\
    3 & 8 & 9 & 15 & 18 & 21 \\
    4 & 10 & 11 & 16 & 20 & 23 \\
    5 & 12 & 13 & 17 & 22 & 24
\end{ytableau}}

\def\PT{
\ytableausetup{boxsize=0.6cm}
\begin{ytableau}
    1 & 5 & 6 & 13 & 17 & 18 \\
    2 & 7 & 8 & 14 & 19 & 20 \\
    3 & 9 & 10 & 15 & 21 & 22 \\
    4 & 11 & 12 & 16 & 23 & 24
\end{ytableau}
}

\def\MT{
\begin{tikzpicture}[scale=0.45]
\draw[style = dashed, <->] (0,0)--(25,0);
\node at (1,-0.5) {\tiny 1};
\node at (2,-0.5) {{\tiny 2}};
\node at (3,-0.5) {{\tiny3}};
\node at (4,-0.5) {{\tiny4}};
\node at (5,-0.5) {{\tiny5}};
\node at (6,-0.5) {\tiny6};
\node at (7,-0.5) {\tiny7};
\node at (8,-0.5) {{\tiny8}};
\node at (9,-0.5) {{\tiny9}};
\node at (10,-0.5) {{\tiny10}};
\node at (11,-0.5) {{\tiny11}};
\node at (12,-0.5) {\tiny12};
\node at (13,-0.5) {\tiny13};
\node at (14,-0.5) {{\tiny14}};
\node at (15,-0.5) {{\tiny15}};
\node at (16,-0.5) {{\tiny16}};
\node at (17,-0.5) {{\tiny17}};
\node at (18,-0.5) {\tiny18};
\node at (19,-0.5) {\tiny19};
\node at (20,-0.5) {{\tiny20}};
\node at (21,-0.5) {{\tiny21}};
\node at (22,-0.5) {{\tiny22}};
\node at (23,-0.5) {{\tiny23}};
\node at (24,-0.5) {{\tiny24}};

\draw[radius=.08, fill=black](1,0)circle;
\draw[radius=.08, fill=black](2,0)circle;
\draw[radius=.08, fill=black](3,0)circle;
\draw[radius=.08, fill=black](4,0)circle;
\draw[radius=.08, fill=black](5,0)circle;
\draw[radius=.08, fill=black](6,0)circle;
\draw[radius=.08, fill=black](7,0)circle;
\draw[radius=.08, fill=black](8,0)circle;
\draw[radius=.08, fill=black](9,0)circle;
\draw[radius=.08, fill=black](10,0)circle;
\draw[radius=.08, fill=black](11,0)circle;
\draw[radius=.08, fill=black](12,0)circle;
\draw[radius=.08, fill=black](13,0)circle;
\draw[radius=.08, fill=black](14,0)circle;
\draw[radius=.08, fill=black](15,0)circle;
\draw[radius=.08, fill=black](16,0)circle;
\draw[radius=.08, fill=black](17,0)circle;
\draw[radius=.08, fill=black](18,0)circle;
\draw[radius=.08, fill=black](19,0)circle;
\draw[radius=.08, fill=black](20,0)circle;
\draw[radius=.08, fill=black](21,0)circle;
\draw[radius=.08, fill=black](22,0)circle;
\draw[radius=.08, fill=black](23,0)circle;
\draw[radius=.08, fill=black](24,0)circle;

\draw[style=thick, color=black] (3,0) arc (0:180: .5cm);
\draw[style=thick, color=black] (4,0) arc (0:180: .5cm);
\draw[style=thick, color=black] (5,0) arc (0:180: .5cm);

\draw[style=thick, color=black] (15,0) arc (0:180: .5cm);
\draw[style=thick, color=black] (16,0) arc (0:180: .5cm);
\draw[style=thick, color=black] (17,0) arc (0:180: .5cm);

\draw[style=thick, color=black] (24,0) arc (0:180: .5cm);

\draw[style=thick] (18,0) arc [start angle=0,end angle=180,x radius=8.5cm, y radius=2cm];

\draw[style=thick, color=black] (8,0) arc (0:180: .5cm);
\draw[style=thick, color=black] (10,0) arc (0:180: .5cm);
\draw[style=thick, color=black] (12,0) arc (0:180: .5cm);

\draw[style=thick] (9,0) arc [start angle=0,end angle=180,x radius=1.5cm, y radius=1cm];
\draw[style=thick] (11,0) arc [start angle=0,end angle=180,x radius=1.5cm, y radius=1cm];
\draw[style=thick] (13,0) arc [start angle=0,end angle=180,x radius=1.5cm, y radius=1cm];

\draw[style=thick] (20,0) arc [start angle=0,end angle=180,x radius=1cm, y radius=1cm];
\draw[style=thick] (21,0) arc [start angle=0,end angle=180,x radius=1cm, y radius=1cm];
\draw[style=thick] (22,0) arc [start angle=0,end angle=180,x radius=1cm, y radius=1cm];
\draw[style=thick] (23,0) arc [start angle=0,end angle=180,x radius=1cm, y radius=1cm];

\end{tikzpicture}
}

\def\loneexample{
\begin{tikzpicture}[scale=0.45]
\draw[style = dashed, <->] (0,0)--(17,0);

\node at (1,-0.5) {\tiny 1};
\node at (2,-0.5) {{\tiny 2}};
\node at (3,-0.5) {{\tiny3}};
\node at (4,-0.5) {{\tiny4}};
\node at (5,-0.5) {{\tiny5}};
\node at (6,-0.5) {\tiny6};
\node at (7,-0.5) {\tiny7};
\node at (8,-0.5) {{\tiny8}};
\node at (9,-0.5) {{\tiny9}};
\node at (10,-0.5) {{\tiny10}};
\node at (11,-0.5) {{\tiny11}};
\node at (12,-0.5) {\tiny12};
\node at (13,-0.5) {\tiny13};
\node at (14,-0.5) {{\tiny14}};
\node at (15,-0.5) {{\tiny15}};
\node at (16,-0.5) {{\tiny16}};

\draw[style=thick] (3,0) arc (0:180: .5cm);
\draw[style=thick] (5,0) arc (0:180: .5cm);
\draw[style=thick] (7,0) arc (0:180: .5cm);

\draw[style=thick] (4,0) arc [start angle=0,end angle=180,x radius=1.5cm, y radius=1cm];
\draw[style=thick] (6,0) arc [start angle=0,end angle=180,x radius=1.5cm, y radius=1cm];
\draw[style=thick] (8,0) arc [start angle=0,end angle=180,x radius=1.5cm, y radius=1cm];

\draw[style=thick] (11,0) arc (0:180: .5cm);
\draw[style=thick] (13,0) arc (0:180: .5cm);
\draw[style=thick] (15,0) arc (0:180: .5cm);

\draw[style=thick] (12,0) arc [start angle=0,end angle=180,x radius=1.5cm, y radius=1cm];
\draw[style=thick] (14,0) arc [start angle=0,end angle=180,x radius=1.5cm, y radius=1cm];
\draw[style=thick] (16,0) arc [start angle=0,end angle=180,x radius=1.5cm, y radius=1cm];

\draw[radius=.08, fill=black](1,0)circle;
\draw[radius=.08, fill=black](2,0)circle;
\draw[radius=.08, fill=black](3,0)circle;
\draw[radius=.08, fill=black](4,0)circle;
\draw[radius=.08, fill=black](5,0)circle;
\draw[radius=.08, fill=black](6,0)circle;
\draw[radius=.08, fill=black](7,0)circle;
\draw[radius=.08, fill=black](8,0)circle;
\draw[radius=.08, fill=black](9,0)circle;
\draw[radius=.08, fill=black](10,0)circle;
\draw[radius=.08, fill=black](11,0)circle;
\draw[radius=.08, fill=black](12,0)circle;
\draw[radius=.08, fill=black](13,0)circle;
\draw[radius=.08, fill=black](14,0)circle;
\draw[radius=.08, fill=black](15,0)circle;
\draw[radius=.08, fill=black](16,0)circle;
\end{tikzpicture}
}

\def\ltwoexample{
\begin{tikzpicture}[scale=0.45]
\draw[style = dashed, <->] (0,0)--(17,0);

\node at (1,-0.5) {\tiny 1};
\node at (2,-0.5) {{\tiny 2}};
\node at (3,-0.5) {{\tiny3}};
\node at (4,-0.5) {{\tiny4}};
\node at (5,-0.5) {{\tiny5}};
\node at (6,-0.5) {\tiny6};
\node at (7,-0.5) {\tiny7};
\node at (8,-0.5) {{\tiny8}};
\node at (9,-0.5) {{\tiny9}};
\node at (10,-0.5) {{\tiny10}};
\node at (11,-0.5) {{\tiny11}};
\node at (12,-0.5) {\tiny12};
\node at (13,-0.5) {\tiny13};
\node at (14,-0.5) {{\tiny14}};
\node at (15,-0.5) {{\tiny15}};
\node at (16,-0.5) {{\tiny16}};

\draw[style=thick, color=OrangeRed] (3,0) arc (0:180: .5cm);
\draw[style=thick, color=OrangeRed] (5,0) arc (0:180: .5cm);
\draw[style=thick, color=OrangeRed] (7,0) arc (0:180: .5cm);

\draw[style=thick, color=OrangeRed] (4,0) arc [start angle=0,end angle=180,x radius=1.5cm, y radius=1cm];
\draw[style=thick, color=OrangeRed] (6,0) arc [start angle=0,end angle=180,x radius=1.5cm, y radius=1cm];
\draw[style=thick, color=OrangeRed] (8,0) arc [start angle=0,end angle=180,x radius=1.5cm, y radius=1cm];

\draw[style=thick, color=BurntOrange] (12,0) arc [start angle=0,end angle=180,x radius=1cm, y radius=1cm];
\draw[style=thick, color=BurntOrange] (13,0) arc [start angle=0,end angle=180,x radius=1cm, y radius=1cm];
\draw[style=thick, color=BurntOrange] (14,0) arc [start angle=0,end angle=180,x radius=1cm, y radius=1cm];
\draw[style=thick, color=BurntOrange] (15,0) arc [start angle=0,end angle=180,x radius=1cm, y radius=1cm];

\draw[style=thick, color=BurntOrange] (10,0) arc (0:180: .5cm);
\draw[style=thick, color=BurntOrange] (16,0) arc (0:180: .5cm);

\draw[radius=.08, fill=black](1,0)circle;
\draw[radius=.08, fill=black](2,0)circle;
\draw[radius=.08, fill=black](3,0)circle;
\draw[radius=.08, fill=black](4,0)circle;
\draw[radius=.08, fill=black](5,0)circle;
\draw[radius=.08, fill=black](6,0)circle;
\draw[radius=.08, fill=black](7,0)circle;
\draw[radius=.08, fill=black](8,0)circle;
\draw[radius=.08, fill=black](9,0)circle;
\draw[radius=.08, fill=black](10,0)circle;
\draw[radius=.08, fill=black](11,0)circle;
\draw[radius=.08, fill=black](12,0)circle;
\draw[radius=.08, fill=black](13,0)circle;
\draw[radius=.08, fill=black](14,0)circle;
\draw[radius=.08, fill=black](15,0)circle;
\draw[radius=.08, fill=black](16,0)circle;
\end{tikzpicture}
}

\def\MTcomponents{
\begin{tikzpicture}[scale=0.45]
\draw[style = dashed, <->] (0,0)--(25,0);
\node at (1,-0.5) {\tiny 1};
\node at (2,-0.5) {{\tiny 2}};
\node at (3,-0.5) {{\tiny3}};
\node at (4,-0.5) {{\tiny4}};
\node at (5,-0.5) {{\tiny5}};
\node at (6,-0.5) {\tiny6};
\node at (7,-0.5) {\tiny7};
\node at (8,-0.5) {{\tiny8}};
\node at (9,-0.5) {{\tiny9}};
\node at (10,-0.5) {{\tiny10}};
\node at (11,-0.5) {{\tiny11}};
\node at (12,-0.5) {\tiny12};
\node at (13,-0.5) {\tiny13};
\node at (14,-0.5) {{\tiny14}};
\node at (15,-0.5) {{\tiny15}};
\node at (16,-0.5) {{\tiny16}};
\node at (17,-0.5) {{\tiny17}};
\node at (18,-0.5) {\tiny18};
\node at (19,-0.5) {\tiny19};
\node at (20,-0.5) {{\tiny20}};
\node at (21,-0.5) {{\tiny21}};
\node at (22,-0.5) {{\tiny22}};
\node at (23,-0.5) {{\tiny23}};
\node at (24,-0.5) {{\tiny24}};

\node at (3.5,-1.25) {{\color{Cerulean}\small$C_2$}};
\node at (9.5,-1.25) {{\color{OrangeRed}\small$C_3$}};
\node at (15.5,-1.25) {{\color{ForestGreen}\small$C_4$}};
\node at (18.25,1.5) {{\color{BurntOrange}\small$C_1$}};

\draw[style=thick, color=Cerulean] (3,0) arc (0:180: .5cm);
\draw[style=thick, color=Cerulean] (4,0) arc (0:180: .5cm);
\draw[style=thick, color=Cerulean] (5,0) arc (0:180: .5cm);

\draw[style=thick, color=ForestGreen] (15,0) arc (0:180: .5cm);
\draw[style=thick, color=ForestGreen] (16,0) arc (0:180: .5cm);
\draw[style=thick, color=ForestGreen] (17,0) arc (0:180: .5cm);

\draw[style=thick, color=BurntOrange] (24,0) arc (0:180: .5cm);

\draw[style=thick, color=BurntOrange] (18,0) arc [start angle=0,end angle=180,x radius=8.5cm, y radius=2cm];

\draw[style=thick, color=OrangeRed] (8,0) arc (0:180: .5cm);
\draw[style=thick, color=OrangeRed] (10,0) arc (0:180: .5cm);
\draw[style=thick, color=OrangeRed] (12,0) arc (0:180: .5cm);

\draw[style=thick, color=OrangeRed] (9,0) arc [start angle=0,end angle=180,x radius=1.5cm, y radius=1cm];
\draw[style=thick, color=OrangeRed] (11,0) arc [start angle=0,end angle=180,x radius=1.5cm, y radius=1cm];
\draw[style=thick, color=OrangeRed] (13,0) arc [start angle=0,end angle=180,x radius=1.5cm, y radius=1cm];

\draw[style=thick, color=BurntOrange] (20,0) arc [start angle=0,end angle=180,x radius=1cm, y radius=1cm];
\draw[style=thick, color=BurntOrange] (21,0) arc [start angle=0,end angle=180,x radius=1cm, y radius=1cm];
\draw[style=thick, color=BurntOrange] (22,0) arc [start angle=0,end angle=180,x radius=1cm, y radius=1cm];
\draw[style=thick, color=BurntOrange] (23,0) arc [start angle=0,end angle=180,x radius=1cm, y radius=1cm];

\draw[radius=.08, fill=black](1,0)circle;
\draw[radius=.08, fill=black](2,0)circle;
\draw[radius=.08, fill=black](3,0)circle;
\draw[radius=.08, fill=black](4,0)circle;
\draw[radius=.08, fill=black](5,0)circle;
\draw[radius=.08, fill=black](6,0)circle;
\draw[radius=.08, fill=black](7,0)circle;
\draw[radius=.08, fill=black](8,0)circle;
\draw[radius=.08, fill=black](9,0)circle;
\draw[radius=.08, fill=black](10,0)circle;
\draw[radius=.08, fill=black](11,0)circle;
\draw[radius=.08, fill=black](12,0)circle;
\draw[radius=.08, fill=black](13,0)circle;
\draw[radius=.08, fill=black](14,0)circle;
\draw[radius=.08, fill=black](15,0)circle;
\draw[radius=.08, fill=black](16,0)circle;
\draw[radius=.08, fill=black](17,0)circle;
\draw[radius=.08, fill=black](18,0)circle;
\draw[radius=.08, fill=black](19,0)circle;
\draw[radius=.08, fill=black](20,0)circle;
\draw[radius=.08, fill=black](21,0)circle;
\draw[radius=.08, fill=black](22,0)circle;
\draw[radius=.08, fill=black](23,0)circle;
\draw[radius=.08, fill=black](24,0)circle;

\end{tikzpicture}
}

\def\rhoMT{
\begin{tikzpicture}[scale=0.3]
\draw[style = dashed, <->] (0,0)--(25,0);

\draw[radius=.08, fill=black](1,0)circle;
\draw[radius=.08, fill=black](2,0)circle;
\draw[radius=.08, fill=black](3,0)circle;
\draw[radius=.08, fill=black](4,0)circle;
\draw[radius=.08, fill=black](5,0)circle;
\draw[radius=.08, fill=black](6,0)circle;
\draw[radius=.08, fill=black](7,0)circle;
\draw[radius=.08, fill=black](8,0)circle;
\draw[radius=.08, fill=black](9,0)circle;
\draw[radius=.08, fill=black](10,0)circle;
\draw[radius=.08, fill=black](11,0)circle;
\draw[radius=.08, fill=black](12,0)circle;
\draw[radius=.08, fill=black](13,0)circle;
\draw[radius=.08, fill=black](14,0)circle;
\draw[radius=.08, fill=black](15,0)circle;
\draw[radius=.08, fill=black](16,0)circle;
\draw[radius=.08, fill=black](17,0)circle;
\draw[radius=.08, fill=black](18,0)circle;
\draw[radius=.08, fill=black](19,0)circle;
\draw[radius=.08, fill=black](20,0)circle;
\draw[radius=.08, fill=black](21,0)circle;
\draw[radius=.08, fill=black](22,0)circle;
\draw[radius=.08, fill=black](23,0)circle;
\draw[radius=.08, fill=black](24,0)circle;

\draw[style=thick, color=black] (2,0) arc (0:180: .5cm);
\draw[style=thick, color=black] (3,0) arc (0:180: .5cm);
\draw[style=thick, color=black] (4,0) arc (0:180: .5cm);

\draw[style=thick, color=black] (14,0) arc (0:180: .5cm);
\draw[style=thick, color=black] (15,0) arc (0:180: .5cm);
\draw[style=thick, color=black] (16,0) arc (0:180: .5cm);

\draw[style=thick, color=black] (23,0) arc (0:180: .5cm);

\draw[style=thick] (24,0) arc [start angle=0,end angle=180,x radius=3.5cm, y radius=2cm];

\draw[style=thick, color=black] (7,0) arc (0:180: .5cm);
\draw[style=thick, color=black] (9,0) arc (0:180: .5cm);
\draw[style=thick, color=black] (11,0) arc (0:180: .5cm);

\draw[style=thick] (8,0) arc [start angle=0,end angle=180,x radius=1.5cm, y radius=1.5cm];
\draw[style=thick] (10,0) arc [start angle=0,end angle=180,x radius=1.5cm, y radius=1.5cm];
\draw[style=thick] (12,0) arc [start angle=0,end angle=180,x radius=1.5cm, y radius=1.5cm];

\draw[style=thick] (19,0) arc [start angle=0,end angle=180,x radius=1cm, y radius=1cm];
\draw[style=thick] (20,0) arc [start angle=0,end angle=180,x radius=1cm, y radius=1cm];
\draw[style=thick] (21,0) arc [start angle=0,end angle=180,x radius=1cm, y radius=1cm];
\draw[style=thick] (22,0) arc [start angle=0,end angle=180,x radius=1cm, y radius=1cm];

\end{tikzpicture}
}

\def\MPT{
\begin{tikzpicture}[scale=0.3]
\draw[style = dashed, <->] (0,0)--(25,0);

\draw[radius=.08, fill=black](1,0)circle;
\draw[radius=.08, fill=black](2,0)circle;
\draw[radius=.08, fill=black](3,0)circle;
\draw[radius=.08, fill=black](4,0)circle;
\draw[radius=.08, fill=black](5,0)circle;
\draw[radius=.08, fill=black](6,0)circle;
\draw[radius=.08, fill=black](7,0)circle;
\draw[radius=.08, fill=black](8,0)circle;
\draw[radius=.08, fill=black](9,0)circle;
\draw[radius=.08, fill=black](10,0)circle;
\draw[radius=.08, fill=black](11,0)circle;
\draw[radius=.08, fill=black](12,0)circle;
\draw[radius=.08, fill=black](13,0)circle;
\draw[radius=.08, fill=black](14,0)circle;
\draw[radius=.08, fill=black](15,0)circle;
\draw[radius=.08, fill=black](16,0)circle;
\draw[radius=.08, fill=black](17,0)circle;
\draw[radius=.08, fill=black](18,0)circle;
\draw[radius=.08, fill=black](19,0)circle;
\draw[radius=.08, fill=black](20,0)circle;
\draw[radius=.08, fill=black](21,0)circle;
\draw[radius=.08, fill=black](22,0)circle;
\draw[radius=.08, fill=black](23,0)circle;
\draw[radius=.08, fill=black](24,0)circle;

\draw[style=thick, color=black] (2,0) arc (0:180: .5cm);
\draw[style=thick, color=black] (3,0) arc (0:180: .5cm);
\draw[style=thick, color=black] (4,0) arc (0:180: .5cm);

\draw[style=thick, color=black] (14,0) arc (0:180: .5cm);
\draw[style=thick, color=black] (15,0) arc (0:180: .5cm);
\draw[style=thick, color=black] (16,0) arc (0:180: .5cm);

\draw[style=thick, color=black] (7,0) arc (0:180: .5cm);
\draw[style=thick, color=black] (9,0) arc (0:180: .5cm);
\draw[style=thick, color=black] (11,0) arc (0:180: .5cm);

\draw[style=thick] (8,0) arc [start angle=0,end angle=180,x radius=1.5cm, y radius=1.5cm];
\draw[style=thick] (10,0) arc [start angle=0,end angle=180,x radius=1.5cm, y radius=1.5cm];
\draw[style=thick] (12,0) arc [start angle=0,end angle=180,x radius=1.5cm, y radius=1.5cm];

\draw[style=thick, color=black] (19,0) arc (0:180: .5cm);
\draw[style=thick, color=black] (21,0) arc (0:180: .5cm);
\draw[style=thick, color=black] (23,0) arc (0:180: .5cm);

\draw[style=thick] (20,0) arc [start angle=0,end angle=180,x radius=1.5cm, y radius=1.5cm];
\draw[style=thick] (22,0) arc [start angle=0,end angle=180,x radius=1.5cm, y radius=1.5cm];
\draw[style=thick] (24,0) arc [start angle=0,end angle=180,x radius=1.5cm, y radius=1.5cm];

\end{tikzpicture}
}

\begin{document}
\title[Identifying orbit lengths for promotion]{Identifying orbit lengths for promotion}
\author{Elise Catania}
\address{University of Minnesota, Church St SE, Minneapolis MN 55455}
\email{catan042@umn.edu}

\author{Jack Kendrick}
\address{University of Washington, Padelford Hall, Seattle WA 98195}
\email{jackgk@uw.edu}

\author{Heather M. Russell}
\address{University of Richmond, Jepson Hall, Richmond VA 23173}
\email{hrussell@richmond.edu}

\author{Julianna Tymoczko}
\address{Smith College, Northampton MA 01063}
\email{jtymoczko@smith.edu}

\thanks{HR was partially supported by an AMS-Simons Research Enhancement Grant for PUI Faculty.  JT was partially supported by NSF DMS grants 1800773, 2054513, and 2349088, and an AWM MERP fellowship.}

\keywords{Young tableaux, jeu de taquin, promotion, web, non-crossing matching}
\subjclass{05E10, 05C10}

\begin{abstract} In this work we study Sch\"utzenberger's promotion operator on standard Young tableaux via a corresponding graphical construction known as $m-$diagrams. In particular, we prove that certain internal structures of SYT are preserved under promotion and correspond to distinct components of $m-$diagrams. By treating these structures as atomic parts of the $m-$diagram, we provide a simple algorithm for computing the promotion orbit length of rectangular SYT. We conclude the paper by applying our results to (column) semi-standard Young tableaux and prove a formula for the promotion orbit lengths of rectangular (column) SSYT.
\end{abstract}

\maketitle

\section{Introduction}

Young tableaux are fundamental combinatorial objects.
With various choices of fillings and rules for manipulations, Young tableaux encode important data and processes from across mathematics \cite{Fulton_1996}. Here, we focus on Sch\"utzenberger's \emph{promotion} operator on \emph{standard Young tableaux} \cite{Schutzenberger1972}. Promotion features in the study of representations and algebraic operations on representations 
\cite{Alexandersson, Haiman, pappe, Pfannerer}, Schubert varieties and the cohomology ring structure induced by Schubert classes \cite{thomas1, thomas2}, and other constructions in algebra and geometry \cite{fontainekamnitzer2014, pechenik2}.

Promotion acts bijectively on the set of standard Young tableaux of a fixed shape and so induces a partition of the set of SYT with a given shape into distinct promotion orbits. A natural question is whether the size of an orbit can be determined from only one representative tableau. Previous work has focused on the {\em promotion order} of the partition, namely the shape of the tableau. The promotion order corresponding to the shape $\lambda$ is the minimal value $N$ such that $P^N(T) = T$ for every tableaux $T$ of shape $\lambda.$ In general, the promotion order is not well-behaved. For instance, the promotion order of the shape $\lambda =(8,6)\vdash 14$ is 7,554,844,752 \cite{striker2}. However, restricting to {\em rectangular} SYT yields much more promising results. A seminal --- and very surprising --- result of Haiman shows that the order of promotion for $m \times n$ rectangular shapes  is $mn$ \cite{Haiman}. In addition, Haiman characterizes all shapes $\lambda$  such that the order of promotion is $|\lambda|$, which he calls generalized staircases.

We can also study the structure and sizes of promotion orbits. An important result due to Rhoades shows that we can get data about promotion orbits for rectangular shapes via cyclic-sieving \cite{reiner, rhoades}. In particular, for each positive integer $n$,  the quantum hook-length polynomial measures the number of tableaux fixed by $n$ applications of the promotion operator.  This has been explored and extended in different contexts, using related operators \cite{pechenik, striker3, striker} or equivariant maps to other sets of objects equipped with a cyclic group action \cite{armstrong, petersen2009promotion, russell, Julianna}.  
 
Yet, little has been done to analyze the sizes and elements of promotion orbits for \emph{specific} standard Young tableaux, despite analogous classification and stability questions in representation theory.  One of the few results holds for rectangular shapes: Purbhoo and Rhee gave a bijection between the tableaux in minimal promotion orbits and elements of the symmetric group \cite{purbhoo}.  

In this paper, we also focus on rectangular shapes. We give a process for computing the size of a tableau's promotion orbit by exploiting the geometry of an object called an $m-$diagram. An $m-$diagram is a union of noncrossing matchings; the connected components of an $m-$ diagram refine the data of a noncrossing partition.  While promotion of rectangular tableaux is not equivariant with rotation of $m-$diagrams, we show that the $m-$diagrams nonetheless encodes important information. In particular, we do the following:
 \begin{enumerate}
\item Define a notion of uniformly equivalent tableaux and investigate promotion via the graphical construction of $m-$diagrams.
 
\item Provide an algorithm to determine the orbit length of a tableau based on geometric symmetries of the connected components in its corresponding $m-$diagram.  In particular, this allows us to analyze orbit lengths when the $m-$diagrams are disconnected (see Theorem \ref{algorithm}).

\item Extend our ideas to (column) semi-standard Young tableaux and provide a formula for the the orbit lengths of these tableaux.
\end{enumerate}

We do not examine the structure of $m-$diagrams deeply in this paper and leave as an open question whether sharper analysis would give more information about promotion orbits.

\subsection*{Outline} In Section \ref{sec:prelim} we provide the necessary background on standard Young tableaux and $m-$diagrams. Our main results on the promotion orbit lengths of rectangular SYT are contained in Section \ref{section: main}. We conclude the paper by applying our results to (column) semi-standard Young tableaux in Section \ref{sec:ssyt}.

\subsection*{Acknowledgements} We would like to thank Meredith Pan and Risa Vandegrift for helpful discussions during the early stages of this work.  
\section{Preliminaries}\label{sec:prelim}

\subsection{Tableaux}{\label{sec: tableaux}} We begin by fixing notations and conventions for tableaux.

\begin{definition}[Standard Young Tableau]
Suppose $\lambda = (\lambda_1 \geq \lambda_2 \geq \cdots \geq \lambda_k)$ is a partition of $n$. A \emph{Young diagram} of shape $\lambda$ is a grid of left- and top-aligned boxes with $\lambda_i$ boxes in the $i^{th}$ row for each $1 \leq i \leq k$. If all parts of the partition $\lambda$ are the same then we call $\lambda$ a rectangular partition and its Young diagram a \emph{rectangular Young diagram}.

A \emph{standard Young tableau (SYT)} is a filling of the Young diagram for $\lambda$ with values $1, \ldots, n$ without repetition such that rows strictly increase left-to-right and columns strictly increase top-to-bottom.  
The \emph{content} of a tableau $T$ is the collection of integers filling its boxes. 
\end{definition}

Given two partitions $\mu$ and $\lambda$ with $\mu_i\leq \lambda_i$ for all $i$, the set of boxes belonging to the Young diagram corresponding to shape $\lambda$ but not $\mu$ yields a {\it skew diagram} of shape $\lambda/\mu$. A {\it skew tableau} is a filling of a skew diagram such that for each box, all entries to the right and below are strictly larger.

Given a tableau $T,$ we denote the number of boxes in $T$ as $|T|$ and the entry in the $i^{th}$ row and $j^{th}$ column of $T$ as $T_{ij}$. Next, we review Sch\"utzenberger's jeu de taquin promotion on SYT, which is a function from SYT of fixed shape $\lambda$ to SYT of the same shape \cite{Schutzenberger1972,Stanley}.

\begin{definition}[Promotion]\label{definition: promotion}
Given an SYT $T$, the \emph{promotion} of $T$ is the SYT $P(T)$ created as follows:
\begin{enumerate}
\item Erase 1 in the top left corner of $T$ and leave an empty box.
\item Given the configuration $\ytableausetup {boxsize=0.5 cm} \begin{ytableau} {} & b \\ a \\ \end{ytableau}$ and $b < a$ then slide $b$ left; else if $a < b$ slide $a$ up.
\item Repeat the above process until there are no nonempty boxes below or to the right of the empty box.
\item Decrement all entries by $1$ and insert the largest entry of $T$ into the empty box.
\end{enumerate}
\end{definition}

\begin{example} Figure~\ref{fig:T and P(T)} gives an example of an SYT $T$ and its promotion $P(T).$ This tableau will serve as our running example.
    \begin{figure}[h]
        \centering
        \begin{tikzpicture}
            \node at (0,0) {\T};
            \node at (-2.5,0) {$T =$};
            \node at (5,0) {\PT};
            \node at (8,0) {$ = P(T)$};

            \draw[style=thick, ->] (2,0) -- (3,0);
        \end{tikzpicture}

        \caption{The tableau $T$ and its promotion $P(T)$}
        \label{fig:T and P(T)}
    \end{figure}
\end{example}

We are particularly interested in the orbits of tableaux under promotion. For a given SYT $T$, we denote the promotion orbit of $T$ by $\cO(T).$ The key observation in our approach is that certain internal structures of SYT are preserved during promotion. To this end, we introduce the following definition.

\begin{definition}[Equivalence]
Suppose $S$ and $T$ are skew tableaux with $|S|=|T|$ and contents $\{s_1, s_2, \ldots, s_n\}, \{t_1, t_2, \ldots, t_n\}.$ We say that $S$ is \textit{equivalent} to $T,$ or $S\equiv T,$ if $S$ and $T$ are the same after each tableau is left justified with the entries relabeled smallest to largest so that the content becomes $\{1, 2, \ldots, n\}.$ Moreover, we say $S$ is \textit{uniformly equivalent} to $T$ if the content of $S$ is $\{k+i_1, k+i_2, \ldots, k+i_n\}$ and the content of $T$ is $\{\ell +i_1, \ell +i_2, \ldots , \ell+i_n\}$ for some $k, \ell$ and $S\equiv T$.  
\end{definition}

\begin{remark}
    Note that if $i_1, \ldots, i_n$ is equal to $1, \ldots, n$ and each row of $S, T$ is no shorter than the row below it, then $S$ and $T$ are uniformly equivalent to an SYT.
\end{remark}

Uniform equivalence is a powerful condition that allows us to treat uniformly proper subtableaux as atomic parts of the corresponding $m$-diagrams (and webs; see Section~\ref{sec: webs}). However, there are contexts in which more generality is useful.

\begin{definition}[Uniformly Proper Subtableau]
Suppose $T$ is a tableau with $n$ rows and $S\subsetneq T$ is a (skew) subtableau. We call $S$ {\em uniformly proper} if:
\begin{enumerate}
    \item $S$ has $n$ rows
    \item $S$ is uniformly equivalent to a standard Young tableau
\end{enumerate}
We call $S$ a uniformly proper {\em rectangular} subtableau if in addition all of its rows have the same length. If $T$ contains no uniformly proper subtableau, we say that $T$ is \textit{minimal.}
\end{definition}

Given two skew tableaux $T$ and $T',$ we write $TT'$ to refer to the tableau resulting from the horizontal concatenation of $T$ and $T'$ when such an operation makes sense. In particular, if $T$ is a SYT that is not minimal, and thus has some uniformly proper rectangular subtableau $S,$ we can write the tableau $T$ as the horizontal concatentation $T = T_1 ST_2$ where $T_1$ and $T_2$ are also subtableaux of $T.$  Figure~\ref{fig:T=T_1ST_2} gives an example. 

Note that if we fix a uniformly proper rectangular subtableau $S$ of $T$, there is a unique pair $T_1, T_2$ that decompose $T$ as $T_1ST_2$.  However, if $T$ contains multiple uniformly proper rectangular subtableaux then there are many ways to decompose $T$ as a horizontal concatenation.  

\begin{figure}[h]
    \centering
    \ytableausetup{boxsize=0.6cm}
\begin{tikzpicture}
\node at (-4, 0) {\T};
\node at (-1.5, 0) {=};
\node at (0,0){\begin{ytableau}
1 & 2 \\
3  \\
4  \\
5  \\
\end{ytableau}};
\node at (2, 0) {\begin{ytableau}
    \none & 6 & 7 & 14 \\
     8 & 9 & 15 \\
    10 & 11 & 16 \\
    12 & 13 & 17
\end{ytableau}};
\node at (4, 0) {\begin{ytableau}
    \none & 19 \\
    18 & 21 \\
    20 & 23 \\
    22 & 24
\end{ytableau}};

\node at (-4, -1.75) {$T$};
\node at (0, -1.75) {$T_1$};
\node at (2, -1.75) {$S$};
\node at (4, -1.75) {$T_2$};
\end{tikzpicture}
    \caption{Decomposition of tableau $T$ with uniformly proper subtableau $S$}
    \label{fig:T=T_1ST_2}
\end{figure}

We will see in Section \ref{section: main} that the problem of determining the orbit length $|\cO(T)|$ of any non-minimal tableau $T$ can be reduced to determining the orbit lengths of a subset of its minimal uniformly proper subtableaux.

\subsection{m-diagrams}{\label{sec: webs}}

We are motivated by open questions about webs, which are planar graphs that represent functions in a diagrammatic category for certain representations of quantum groups $U_q(\mathfrak{sl}_n)$ \cite{kuperberg1996spiders}. 
There is a rich literature on how promotion of tableaux corresponds to rotation of webs in the cases $n=2$ and $n=3$ \cite{petersen2009promotion, russell, Julianna}; 
but this is constrained for larger $n$ by our limited understanding of webs in that case (though see \cite{gaetz2023rotation} for a rotation-invariant basis when $n=4$).

For this reason, we instead focus on an intermediary between webs and tableaux: $m$-diagrams, which are collections of arcs that satisfy certain noncrossing conditions.

\begin{definition}[$m$-Diagram] \label{definition: m-diagrams}
A \emph{matching} $\mathcal{M}$ on the set $\{1, 2, \ldots, n\}$ is a collection of pairs 
\[\mathcal{M} = \{(i_1, j_1), (i_2, j_2), \ldots, (i_k,j_k)\} \subseteq \{1, 2, \ldots, n\} \times \{1, 2, \ldots, n\}\]
such that each element of $\{1,\ldots, n\}$ occurs in at most one pair. We often refer to the pair $(i,j)$ as an \emph{arc} and assume in our notation that $i \leq j$.  We say that an integer $i'$ is \emph{on the arc} $(i,j)$ if $i' \in \{i,j\}$ and \emph{below the arc} $(i,j)$ if $i<i'<j$.

A matching $\mathcal{M}$ is:
\begin{itemize}
    \item \emph{perfect} if every number $1, 2, \ldots, n$ is used on an arc;
    \item \emph{crossing} if it contains two arcs $(i,j), (i',j') \in \mathcal{M}$ with $i<i'<j<j'$;
    \item \emph{noncrossing} if it is not crossing;
    \item \emph{with repetition} if it contains at least one arc of the form $(i,i)$;
    \item \emph{standard} if every integer $i'$ below an arc $(i,j)$ is itself on an arc.
\end{itemize}
An \textit{$m$-diagram} on the set $\{1, 2, \ldots, n\}$ is the union $\bigcup_{i=1}^r \mathcal{M}_i$ where each $\mathcal{M}_i$ is a noncrossing matching.  
\end{definition}

Note that an $m$-diagram may not itself be a noncrossing matching since arcs from one matching are allowed to cross arcs from another matching.  An $m$-diagram may have arcs of the form $(i,i)$ but is itself not a multiset, so each arc appears at most once in an $m$-diagram regardless of how many different matchings contain it.

We represent matchings and $m$-diagrams by drawing semicircular arcs in the upper half plane between the points $i$ and $j$ on the $x$-axis, as in Figure~\ref{fig:T and M_T}. The next lemma demonstrates a simple bijection between SYT and $m-$diagrams, the proof of which can be found in \cite[Lemma 1]{russell} and \cite[Prop. 2.4]{Julianna}. We include a sketch of the main arguments for completeness.

\begin{lemma} \label{lemma: creating m-diagram from tableau}
Let $T$ be an SYT with $r$ rows and content $\{1, \ldots, n\}$.  For each $1\leq i \leq r-1$ the following recursive process constructs a noncrossing matching $\mathcal{M}_i$ on the integers filling rows $i$ and $i+1$ of $T$:
\begin{enumerate}
    \item If row $i$ of $T$ has $\lambda_i$ boxes then denote them by $t_i(1), t_i(2), \ldots, t_i(\lambda_i)$ and similarly for row $i+1$.
    \item Create an arc $(t_i(j), t_{i+1}(1))$ where 
    \[j = \max\{1 \leq s \leq \lambda_i: t_i(s) \leq t_{i+1}(1)\}.\]
    \item If $t_{i+1}(1), \ldots, t_{i+1}(\ell)$ are all on arcs then create an arc $(t_i(j), t_{i+1}(\ell+1))$ where 
    \[j = \max\{1 \leq s \leq \lambda_i: t_i(s) \leq t_{i+1}(\ell+1) \textup{ and } (t_i(s),t_{i+1}(j')) \textup{ is not an arc for any } j' \leq \ell\}.\]
\end{enumerate}
\end{lemma}

\begin{proof}
For each row $i+1$ and each entry $t_{i+1}(\ell)$ in that row, there are at least $\ell$ integers in row $i$ of value at most $t_{i+1}(\ell)$ because the entries of $T$ are increasing row- and column-wise.  These integers are all distinct because the values in each row of $T$ increase strictly and so the recursive process constructs an arc for each box in row $i+1$ of $T$.   

Suppose $(t_i(\ell), t_{i+1}(\ell'))$ is an arc in the matching $\mathcal{M}_i$.  If $j$ is an integer on row $i+1$ with $j < t_{i+1}(\ell')$ then $j$ is the endpoint of the arc created in an earlier step by the recursive construction to build $\mathcal{M}_i$.   Thus $\mathcal{M}_i$ is a standard matching on the set of integers filling rows $i$ and $i+1$ of $T$.  Moreover the arc $(j',j)$ satisfies $t_i(\ell)<j'$ since otherwise the recursive construction would have paired $j$ with $t_i(\ell)$.  Thus the arc $(j',j)$ does not cross $(t_i(\ell), t{i+1}(\ell'))$. Similarly, if $j$ is an integer on row $i$ with $t_i(\ell) < j < t_{i+1}(\ell')$ then $j$ must be on an arc $(j,t_{i+1}(\ell''))$ for some $\ell''<\ell'$ else $j=t_i(\ell)$ by construction. So $\mathcal{M}_i$ is noncrossing. This completes the proof.
\end{proof}

 In the following result, all of the proofs are immediate from the previous lemma together with the definitions.

\begin{corollary} \label{corollary: function varphi from tabs to m-diagrams}
The algorithm in Lemma~\ref{lemma: creating m-diagram from tableau} defines a function $\varphi(T)=M_T$ from $r$-row SYT to $m$-diagrams built from $r-1$ matchings.  In addition, all of the following hold:
\begin{itemize}
    \item The number of arcs in matching $\mathcal{M}_i$ equals the number of boxes in row $i+1$ of $T$.  
    \item Each matching $\mathcal{M}_i$ is without repetition.
    \item The matching $\mathcal{M}_i$ is standard with respect to the entries in rows $i$ and $i+1$ of $T$.
    \item If $T$ is a standard Young tableau on a rectangular Young diagram then $\varphi(T)=M_T$ is the union of $r-1$ matchings $\mathcal{M}_i$ of the same cardinality.  Moreover every integer $1, 2, \ldots, |T|$ is one of three types: both the end of an arc in $\mathcal{M}_i$ and the start of an arc in $\mathcal{M}_{i+1}$, or the start of an arc on $\mathcal{M}_1$, or the end of an arc in $\mathcal{M}_{r-1}$.
\end{itemize}  
\end{corollary}
Figure~\ref{fig:T and M_T} gives an example of an $m-$diagram corresponding to an SYT.

\begin{figure}[H]
    \centering
    \begin{tikzpicture}[scale=0.95]
        \node at (0,0) {\T};
        \node at (0, -1.75) {$T$};
        \node at (9,0) {\MT};
        \node at (8.5,-1.75) {$\phi(T)$};
        \draw[style=thick, ->] (2.2,0)--(3,0);
    \end{tikzpicture}
    \caption{Tableau $T$ and its corresponding $m-$diagram $\phi(T)$}
    \label{fig:T and M_T}
\end{figure}
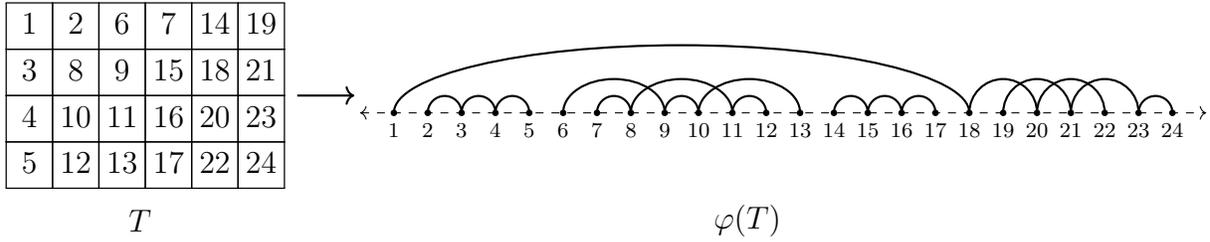

We define a rotation operation on $m$-diagrams algebraically as follows.  Intuitively, it consists of connecting the endpoints of the boundary to form a circle, rotating the circle one step clockwise, and then disconnecting the circle to again form a boundary line.

\begin{definition}[Rotation] \label{definition: rotation}
Denote by $\rho: \{1, 2, \ldots, n\} \rightarrow \{1, 2, \ldots, n\}$ the cyclic permutation given by $\rho(i)=i-1 \mod n$.  Then $\rho$ induces a map from the set of matchings on $\{1, 2, \ldots, n\}$ to itself by sending each arc $(a,b) \mapsto (\rho(a),\rho(b))$.  We call this map \emph{rotation} and denote the image of a matching $\mathcal{M}$ under rotation by $\rho(\mathcal{M})$.  
\end{definition}
Figure~\ref{fig:promotion-vs-rotation} shows the rotation $\rho(\phi(T))$ of the $m-$diagram $\phi(T)$ from Figure~\ref{fig:T and M_T}.  

The following two observations about rotation are almost immediate from the definitions.

\begin{lemma} \label{lemma: rotation preserves perfect and noncrossing matchings}
If $\mathcal{M}$ is a perfect matching then so is the rotation $\rho(\mathcal{M})$.  If $\mathcal{M}$ is a noncrossing matching then so is the rotation $\rho(\mathcal{M})$.
\end{lemma}

\begin{proof}
Rotation sends perfect matchings to perfect matchings because it permutes the endpoints of the arcs.

Suppose $\mathcal{M}$ is noncrossing. Every arc in $\cM$ without endpoint 1, is simply shifted one unit to the left by $\rho$. Therefore, if two arcs cross in $\rho(\cM)$ one of them must be $\rho((1, k)) = (k-1, n).$ Say some arc $(i,j) \in \rho(\mathcal{M})$ crosses $(k-1,n)$. This implies $i<k-1 < j< n$ and also $1<i+1<k<j+1$. It follows that $\rho^{-1}((i,j))=(i+1, j+1)$ and $(1,k)$ cross in  $\mathcal{M}.$ This contradiction proves the claim.
\end{proof}

\begin{figure}[h]
    \centering
    \begin{tikzpicture}
        \node at (0, 0) {\MT};
        \node at (-4, -3.75) {\rhoMT};
        \node at (4, -3.825) {\MPT};

        \node at (0, -1) {$\phi(T)$};
        \node at (-4, -3) {$\rho(\phi(T))$};
        \node at (4, -3) {$\phi(P(T))$};

        \draw[->, style=dashed] (-0.5, -1) arc [start angle=90, end angle = 180, x radius = 3.5, y radius=1.5];
        \draw[->, style=dashed] (0.5, -1) arc [start angle=90, end angle = 180, x radius = -3.5, y radius=1.5];
    \end{tikzpicture}
    \caption{The rotation $\rho(\phi(T))$ compared to the $m-$diagram $\phi(P(T)).$}
    \label{fig:promotion-vs-rotation}
\end{figure}
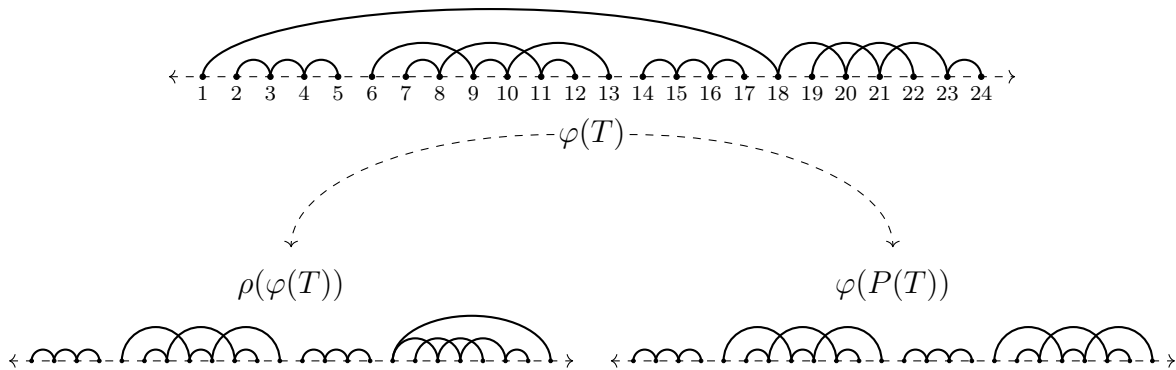

Just as rotation preserves the properties of being standard and noncrossing, we would like rotation to respect promotion, in the sense that rotating the $m$-diagram for SYT $T$ produces the $m$-diagram for the promotion $P(T)$. This does indeed hold for two-row rectangular SYT and for $\mathfrak{sl}_3$ webs coming from $m$-diagrams for three-row rectangular SYT \cite{petersen2009promotion, Julianna}. In general, however, promotion and rotation do not commute. For example, it is not true that the $m-$diagram $\phi(P(T))$ is equal to $\rho(\phi(T))$ for the tableau $T$ in our running example, see Figure \ref{fig:promotion-vs-rotation}. Moreover, in general, $\rho(\varphi(T))\neq \varphi(T')$ for any SYT $T'$. The next section describes a different way to connect rotation to promotion. The arguments in Section \ref{section: main} utilize the idea of \textit{sub-diagrams} and \textit{components} of an $m-$diagram, which we introduce here.

\begin{definition}[Subdiagram and Component]\label{def:component} Consider an $m-$diagram $M = \{(a_i, b_i)\}_{i=1}^r$ where each $(a_i, b_i)$ is an arc. Suppose $C\subsetneq M$ is a nonempty proper subcollection with $C=\{(a_{i_k}, b_{i_k})\}_{k=1}^s$. We say $C$ forms a \textit{sub-diagram} of $M$ if 
\begin{itemize}
    \item whenever $a$ is an endpoint of some arc in $C,$ all arcs in $M$ with $a$ as an endpoint are also in $C$ and
    \item no arc in $M$ but not $C$ crosses any arc in $C, $ namely no arc $(a_j, b_j)$ with $j\notin\{i_1, \ldots, i_s\}$ crosses an arc in $C.$ 
\end{itemize}  The sub-diagram is \textit{uniform} if its endpoints are all adjacent and a \textit{component} if it has no proper subcollection that forms a sub-diagram. Two sub-diagrams $C, C'$ are \textit{equivalent} if $\rho^N(C) = C'$ for some $N.$
\end{definition}

\begin{example}
    We elucidate Definition \ref{def:component} by illustrating the components of the $m-$diagram $\phi(T)$ from our running example, see Figure~\ref{fig:components}.
    \begin{figure}[h]
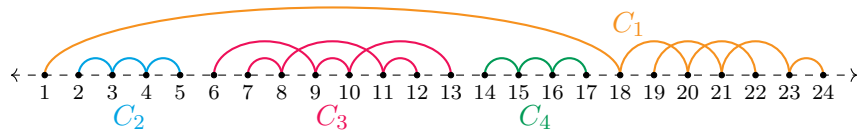

        \centering
        \MTcomponents
        \caption{Components of $\phi(T)$}
        \label{fig:components}
    \end{figure}
    The components of the $m-$diagram $\phi(T)$ are $C_1, C_2, C_3,$ and $C_4.$ The components $C_2, C_3, C_4$ are uniform, whereas $C_1$ is non-uniform. Note that the components $C_2$ and $C_4$ are equivalent since $\rho^{12}(C_4) = C_2.$ 
    Any union of components forms a sub-diagram of $\phi(T).$ For example, $C_2\cup C_3$ forms a uniform sub-diagram, whereas both $C_1\cup C_3$ and $C_2\cup C_4$ form non-uniform sub-diagrams.
\end{example}

\section{Determining Orbit Lengths}{\label{section: main}}
Our main result is an identification between the orbit lengths of rectangular SYT and the rotational symmetry of their corresponding $m-$diagrams. This identification allows us to explicitly compute the orbit length of a non-minimal tableau by only considering certain subtableaux. To this end, we provide a simple algorithm to compute orbit lengths of non-minimal SYT.

{\bf Throughout this section we fix a non-minimal \emph{rectangular} SYT $T$ with uniformly proper rectangular subtableau $S.$} It will often be convenient to consider $T$ as the horizontal concatenation $T_1 S T_2$ where the subtableau $T_1$ is non-empty, i.e. $|T_1|\geq 1.$

We begin by building a correspondence between uniformly proper rectangular subtableaux of $T$ and uniform sub-diagrams of its corresponding $m-$diagram $\phi(T).$ First, we show that the arcs generated by the algorithm in Lemma \ref{lemma: creating m-diagram from tableau} either have both end points in $S$ or neither end point in $S$.

\begin{lemma}{\label{lem: endpoints in S}} Suppose $S$ is a uniformly proper rectangular subtableau of $T.$ If the $m-$diagram $\varphi(T)$ contains an arc $(a, b)$ with either $a$ an entry of $S$ or $b$ an entry of $S,$ then both $a$ and $b$ are entries of $S.$
\end{lemma}
\begin{proof}
    Write $T$ as the horizontal concatenation $T= T_1 S T_2$ and suppose that $a = T_{ij}.$ Then, since $(a,b)$ is an arc in $\varphi(T),$ we must have that $b$ is on row $i+1$ of $T.$
    
    First, suppose that $b$ is an entry of $S.$ Since every entry of $T_2$ is greater than every entry of $T_1S,$ we must have that $a$ is an entry of $T_1S.$ Note that since $T_1$ is a proper subtableau of $T,$ the length of row $i$ of $T_1$ is at least the length of row $i+1$ of $T_1.$ Thus, when the algorithm in Lemma \ref{lemma: creating m-diagram from tableau} is creating arcs that begin with entries in row $i+1$ of $T,$ every entry on row $i+1$ of $T_1$ will be matched to an entry on row $i$ of $T_1.$ Since we also have that $S$ is rectangular, for each entry on row $i+1$ of $S,$ there is at least one entry on row $i$ of $S$ that is yet to be matched by the algorithm in Lemma \ref{lemma: creating m-diagram from tableau}. Moreover, since $S$ is uniformly equivalent to an SYT, for each entry on row $i+1$ of $S,$ there is at least one entry on row $i$ of $S$ that has lower value. As each entry of $S$ is greater than every entry of $T_1,$ it follows that each entry on row $i+1$ of $S$ is matched to an entry on row $i$ of $S.$ Thus, $a$ is an entry of $S.$

    Now, suppose $a$ is an entry of $S.$ Since each entry of $T_1$ is less than every entry of $S,$ we must have that $b$ is an entry of $ST_2.$ By the above argument, every entry on row $i+1$ of $S$ is matched to some entry on row $i$ of $S.$ Since $S$ is rectangular, the number of arcs beginning at an entry on row $i+1$ of $S$ is equal to the number of arcs ending at an entry on row $i$ of $S.$ It follows that $b$ is an entry on row $i+1$ of $S.$
\end{proof}

We can now identify uniformly proper rectangular subtableaux of $T$ with uniform sub-diagrams of the $m-$diagram $\phi(T).$ Moreover, by applying this correspondence recursively, we see that minimal uniformly proper subtableaux are in correspondence with uniform components of $\phi(T).$

\begin{lemma}{\label{lem: components subtableaux bijection}}
    Let $M=\varphi(T).$ We have the following correspondences:
    \begin{enumerate}
        \item The set of uniformly proper rectangular subtableaux of $T$ is in bijection with the uniform sub-diagrams of $M.$
        \item The set of minimal uniformly proper rectangular subtableaux of $T$ is in bijection with the uniform components of $M.$
    \end{enumerate}
\end{lemma}
\begin{proof}
Note that (2) is a simple consequence of (1) and so we only prove (1).

Fix a uniformly proper rectangular subtableau $S$ of $T$ and let $C_S$ be the collection of arcs in $M$ with an endpoint that is an entry of $S.$ By Lemma \ref{lem: endpoints in S}, any arc in $C_S$ must have both endpoints in $S.$ Note that this means that no arc of $M$ with an endpoint outside of $S$ can cross an arc in $C_S.$ Thus, we have that $C_S$ is a sub-diagram of $M.$ As $S$ is a \textit{uniformly} proper subtableau of $T,$ it follows that the endpoints of $C_S$ are all adjacent and so $C_S$ is a uniform sub-diagram of $M.$

Now, fix a uniform sub-diagram $C$ of $M.$ Suppose that the endpoints of arcs in $C$ are $\{k, k+1, \ldots, k+\ell\}$ and consider the subtableau $S$ of $T$ with entries $\{k, k+1, \ldots, k+\ell\}.$ We show that $S$ is a uniformly proper rectangular subtableau. All parts of this proof amount to showing that violating any part of the definition of a uniformly proper rectangular subtableau is equivalent to having an arc in $M$ with one endpoint an entry in $S$ and the other endpoint not an entry in $S$.
\begin{enumerate}
    \item Suppose $S$ has $j$ rows and $T$ has $n$ rows. If $j<n$ then some entry in the $(j+1)^{th}$ row of $T$ is connected to an entry $i$ in $S$ by an arc in $M$. This contradicts the hypothesis that $C$ is a sub-diagram unless $j=n$ as desired.
    \item Similarly, if any row in $S$ is shorter than the row immediately preceding it, there must be an arc in $M$ that connects a boundary vertex outside of $k+1, k+2, \ldots, k+\ell$ to a boundary vertex in the set. Thus, all rows in $S$ must have the same length.
    \item  Suppose that when $S$ is left-justified, it is not standard. Since $S$ is a subtableau of $T,$ the rows of $S$ are increasing and thus if $S$ is not standard when left-justified, there must be some column that is not increasing. Choose $j$ to be the smallest value such that the $j^{th}$ column of $S$ is not increasing. Then for some row $i$ we have that $S_{ij}<S_{(i-1)j}$ and for all $j'<j,$ the columns of $S$ are increasing. As entries to the left of $S_{ij}$ are of lower value and are in increasing columns, the algorithm in Lemma \ref{lemma: creating m-diagram from tableau} matches each entry $S_{ij'}$ for $j'<j$ to some entry on row $i-1$ of $S.$ However, we then have that the algorithm cannot match $S_{(i+1)j}$ to any entry on row $i$ of $S.$ Thus, the arc starting at $S_{ij}$ in $M$ has an endpoint that is not an entry of $S.$ This contradicts the definition of $S.$
    
    Let $S'$ be the tableau obtained by left-justifying $S$ and say $S'$ is not standard. Since $S$ is a subtableau of $T$, the rows of $S$ (hence $S'$) are increasing. Therefore there is a column of $S'$ that is not increasing. Let $j$ be the smallest such column. Then for some row $i$ we have $S'_{ij}<S'_{(i-1)j}$. This means there are fewer than $j$ entries in the $(i-1)^{st}$ row of $S'$ (hence $S$) that are smaller than $S'_{i,j}$, so the algorithm in Lemma \ref{lemma: creating m-diagram from tableau} cannot match $S'_{ij}$ to any entry on row $i-1$ of $S.$ Thus, the arc starting at $S'_{ij}$ in $M$ has an endpoint that is not an entry of $S.$ This contradicts the definition of $S.$ 
\end{enumerate}
This proves that $S$ must be a uniformly proper rectangular subtableau and therefore the map between these subtableaux and uniform sub-diagrams is a bijection.
\end{proof}

Given the correspondence between the internal structure of $T$ and the geometric structure of $\phi(T),$ we now consider how the promotion operation on $T$ affects these structures. The following result shows that the uniformly proper subtableaux of $T$ are preserved under the action of promotion $T.$ Our key insight here is that the ``sliding path'' of the empty box during promotion can only move horizontally through a uniformly proper rectangular subtableau, as in the schematic of Figure \ref{fig: sliding path of b}.

\begin{figure}[h]
    \centering
    \begin{tikzpicture} [scale=.55]
    \node at (.5,5.5) {b};
    \node at (2.75,4.5) {$T_1$};
    \node at (6.5,3.5) {$S$};
    \node at (9,1) {$T_2$};
\node at (5.5,-.5) {$T$};

\draw[thick,densely dotted] (0,5)--(1,5)--(1,6);
\draw[thick, densely dotted] (10,0)--(10,1)--(11,1);

\draw[thick](0,0)--(0,6)--(11,6)--(11,0)--(0,0);
\draw[thick] (1,0)--(1,2)--(3,2)--(3,3)--(4,3)--(4,5)--(5,5)--(5,6);
\draw[thick] (7,0)--(7,2)--(9,2)--(9,3)--(10,3)--(10,5)--(11,5);
\begin{scope}[thick,color=Cerulean, decoration={markings,mark=at position 0.5 with {\arrow{>}}}]
\draw[postaction={decorate}] (0.75,5.5)--(1.5,5.5);
\draw[postaction={decorate}] (1.5,5.5)--(1.5,3.5);
\draw[postaction={decorate}] (1.5,3.5)--(2.5,3.5);
\draw[postaction={decorate}] (2.5,3.5)--(2.5,2.5);
\draw[postaction={decorate}] (2.5,2.5)--(10.5,2.5);
\draw[postaction={decorate}] (10.5,2.5)--(10.5,.5);
\end{scope}
    \draw [-stealth](11.5,3) -- (12.5,3);
\node at (15,4) {$T'_1$};
\node at (18.5,3) {$S'$};
\node at (22,1.5) {$T'_2$};
\node at (18.5,-.5) {$P(T)$};

\draw[thick](13,0)--(13,6)--(24,6)--(24,0)--(13,0);
\draw[thick] (14,0)--(14,2)--(15,2)--(15,3)--(17,3)--(17,5)--(18,5)--(18,6);
\draw[thick] (20,0)--(20,2)--(21,2)--(21,3)--(23,3)--(23,5)--(24,5);
    \end{tikzpicture}
    \caption{Sliding path of $b$ through $T$ and structure of $P(T)$}
    \label{fig: sliding path of b}
\end{figure}

\begin{lemma}{\label{lem: promotion preserves subtableau}} Let $T$ be an SYT with uniformly proper rectangular subtableau $S$ so that $T=T_1ST_2$ with $|T_1|\geq 1.$ Then, $P(T) = T_1'S'T_2'$ where $T_1'T_2'\equiv P(T_1T_2)$ and $S'$ is a uniformly proper rectangular subtableau of $P(T)$ that is uniformly equivalent to $S.$
\end{lemma}

\begin{proof}
Consider the sliding path of the empty box $b$ during the promotion of $T$.

Each element of $T_1$ is less than every element of $S$ so $b$ can only slide into $S$ when its neighbours to the right and below are both in $S$ or outside of the tableau. Thus, $b$ can only slide into $S$ when it reaches the bottom row of a rectangular sub-shape of $T_1$. If $b$ reaches the bottom row of $T_1$ then $b$ can only slide horizontally in $S$. If $b$ does not reach the bottom row of $T_1$ then consider its two neighbours in $S$ as shown in Figure \ref{fig: comparing entries as b slides}. Since $S$ is uniformly proper, the entry in the box below $b$ is greater than the entry in the box to the right of $b$ and thus $b$ slides to the right. The same argument shows that $b$ only slides horizontally until it is either at the end of the row or at least one of its neighbours is in $T_2$. Suppose $b$ has one neighbour in $S$ and one neighbour in $T_2$. Each row of $S$ is the same length so the box below $b$ is in $T_2$.  All entries in $T_2$ are larger than the entries in $S$ so $b$ continues to slide horizontally. Thus $b$ slides horizontally until both of its neighbours are in $T_2$ or until it is at the end of the row, along a path like that in the schematic of Figure \ref{fig: sliding path of b}.

We have proven that $b$ only slides horizontally in $S$.  This means that the only way promotion changes $S$ is by sliding one row to the left by one box and decreasing entries by 1. Call the resulting tableau $S'$.  Then $S'$ is equivalent to $S$ by construction. Moreover, they are uniformly equivalent because 1 is subtracted from each entry of $S$ to form $S'$.

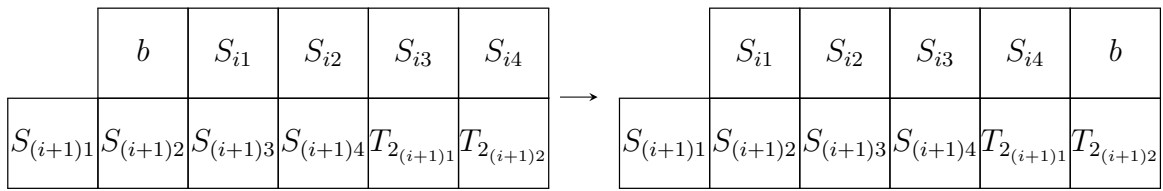
\begin{figure}[H]
    \centering
    \ytableausetup {boxsize=1.182cm}
    \ytableausetup{notabloids}
\begin{ytableau}
\none & b & S_{i1} & S_{i2} & S_{i3} & S_{i4} \\
S_{(i+1)1} & S_{(i+1)2} & S_{(i+1)3} & S_{(i+1)4} & T_{2_{(i+1)1}} &  T_{2_{(i+1)2}}
\end{ytableau}
\begin{tikzpicture}
\draw [-stealth](10,2.75) -- (10.5,2.75);
\end{tikzpicture}
\ytableausetup{notabloids}
\begin{ytableau}
\none & S_{i1} & S_{i2} & S_{i3} & S_{i4} & b\\
S_{(i+1)1} & S_{(i+1)2} & S_{(i+1)3} & S_{(i+1)4} & T_{2_{(i+1)1}} & T_{2_{(i+1)2}}
\end{ytableau}
\caption{Comparing entries as $b$ slides through $S$}
    \label{fig: comparing entries as b slides}
\end{figure}

We have shown $S'$ is a uniformly proper rectangular subtableau of $P(T)$ and so we can write $P(T)$ as the horizontal concatenation $P(T)=T_1'S'T_2'.$ We now show $T_1'T_2'\equiv P(T_1T_2)$ by again considering the sliding path of $b.$ Let $P_1$ denote the part of the sliding path of $b$ that is in $T_1$, $P_2$ the part of the path in $S$, and $P_3$ the part of the path in $T_2$.

Consider the tableau $T_1T_2$ and the sliding path of the empty box $b$ during its promotion. Since each element of $T_1$ is smaller than every element of $T_2$ the empty box $b$ slides into $T_2$ when its neighbours to the right and below are in $T_2$ or outside of the tableau entirely. Promotion is defined from pairwise comparisons so the sliding path of $b$ in $T_1$ is exactly $P_1$. After sliding along $P_1$ the empty box $b$ has neighbors $c_1, c_2$, to the right and below $b$ respectively, that are either in $T_2$ or outside of the tableau. Since $b$ only slides horizontally in $S$ we know that after sliding along the path $P_2$ in $T$, the same entries $c_1, c_2$ neighbour $b$. Thus, the remainder of sliding path of $b$ in $T_1T_2$ is $P_3$ and so the whole sliding path is $P_1P_3$.  It follows that $T_1'T_2'\equiv P(T_1T_2)$ as desired.
\end{proof}

Thus, we have that the uniformly proper rectangular subtableaux of $T$ that do not contain $1$ are preserved during promotion, though their entries are all decremented by 1. A simple consequence of this result along with the correspondence in Lemma \ref{lem: components subtableaux bijection} is that uniform subdiagrams of the $m-$diagram $\phi(T)$ that do not contain arcs ending at $1$ are also in some sense preserved during promotion: for each uniform component $C$ of $\phi(T)$ that does not have 1 as an endpoint, $\rho(C)$ is a uniform subdiagram of $\phi(P(T)).$

\begin{corollary}\label{cor: left translation}
    If $C$ is a uniform sub-diagram of $\varphi(T)$ that does not contain the vertex 1, $\varphi(P(T))$ contains the uniform sub-diagram $\rho(C),$ where $\rho$ is the rotation operation defined in Section \ref{sec: webs}.
\end{corollary}
\begin{proof}
By the correspondence given in Lemma \ref{lem: components subtableaux bijection}, the uniform sub-diagram $C$ corresponds to a uniformly proper subtableau $T_C$ of $T.$ As $C$ does not contain an arc that ends at 1, we have that 1 is not an entry of $T_C.$ Thus, $T$ can be written as $T=T_1T_CT_2$ with $|T_1|\neq0.$ By Lemma \ref{lem: promotion preserves subtableau}, $P(T)$ contains the subtableau $T_C'$ whose entries are the same as $T_C$ with one subtracted from each entry. Again by the correspondence in Lemma \ref{lem: components subtableaux bijection}, $T_C'$ corresponds to a uniform sub-diagram $C'$ of $\varphi(P(T)).$ As $T_C$ and $T_C'$ are uniformly equivalent it is clear from the algorithm given in Section \ref{sec: webs} that entries of $T_C'$ will be matched in the same order as the entries of $T_C$ and thus the arcs in $C'$ will simply be the left translation of the arcs of $C.$ It follows that $C' = \rho(C),$ concluding the proof.
\end{proof}

Our previous results show that the only component of the $m-$diagram $\phi(T)$ that is \emph{not} preserved during promotion is the component $C$ that has an arc ending at $1.$ We wish to describe how the structure of $\phi(T)$ evolves as the promotion operator is applied to $T$. Motivated by this, we examine how promotion alters substructures of $T$ that correspond to the components of $\phi(T).$ 

Note that for each component $C$ of $\phi(T)$ we can construct a rectangular tableau $T_C$ whose content we denote  $\partial C$ and whose $m-$diagram is $C$ itself.  In fact, $T_C$ comes from left justifying the collection of boxes in $T$ with entries $\partial C$. The tableau $T_C$ is necessarily a collection of boxes inside $T$ so we now characterize how the structure of $T_C$ within $T$ changes as we apply the promotion operator.

\begin{definition}[Component Promotion]\label{def:component-promotion}
    Let $C$ be a component of $\phi(T)$ and construct the corresponding tableau $T_C.$ With a slight abuse of language and notation, we define the promotion $P(T_C)$ of $T_C$ as follows:

    \begin{enumerate}
        \item If $1\notin\partial C,$ form $P(T_C)$ by decrementing each entry of $T_C$ by 1.
        \item If $1\in\partial C,$ form $P(T_C)$ by
        \begin{enumerate}
            \item removing 1 from the top left corner of $T_C,$
            \item performing the sequence of sliding moves outlined in Definition \ref{definition: promotion},
            \item decrementing each entry of $T_C$ by 1,
            \item filling the empty box with $|T|.$
        \end{enumerate}
    \end{enumerate}
\end{definition}

\begin{lemma}\label{lem:component-promotion-tableau}
    For each component $C$ of $\phi(T),$ left-justifying the collection of boxes with content $\rho(\partial C)$ in $P(T)$ produces the tableau $P(T_C)$.
\end{lemma}
\begin{proof}
    Let $S$ be the subtableau of $T$ consisting of all boxes of $T$ with entries $j$ such that $\min\{i\in \partial C\}\leq j\leq\max\{i\in\partial C\}.$ Note that $S$ is the minimal uniformly proper rectangular subtableau of $T$ the contains each element of $\partial C$ so, for instance, if $\partial C$ contains both $1$ and $n$ then $S$ will be all of $T$.  The collection of boxes in $S$ with content $\partial C$ form the tableau $T_C$ once left justified.
    
    First, suppose $1\notin \partial C$ and so $1\notin S.$ By Lemma \ref{lem: promotion preserves subtableau}, subtracting 1 from each entry of $S$ yields a uniformly proper rectangular subtableau of $P(T).$ In particular, the boxes with content $\rho(\partial C)$ within $P(T)$ form the tableau $P(T_C)$ as defined in Definition \ref{def:component-promotion}.

    Now, suppose $1\in\partial C.$ Then, we can write $T$ as the horizontal concatenation $T= SR$ where $R$ is the uniformly proper rectangular subtableau of $T$ with all entries greater than the maximum element of $\partial C.$ By Lemma \ref{lem: promotion preserves subtableau}, the promotion $P(T)$ can be written as the horizontal concatenation $P(T) = S_1R'S_2$ where $R'$ is the subtableau of $P(T)$ formed by subtracting 1 from each entry of $R$ and $S_2$ consists of only the box containing $n.$ Moreover, we have the equivalence $S_1S_2\equiv P(S).$ Note that replacing the maximum entry of $P(S)$ with the maximum entry $n$ of $T$  yields $S_1S_2.$
    
    If $C$ is a uniform component of $\phi(T),$ then $S=T_C$ and so $S_1S_2 = P(T_C)$ as defined in Definition \ref{def:component-promotion}. If $C$ is not a uniform component of $\phi(T)$ then the boxes in $S$ whose entries are not in $\partial C$ form $U,$ a uniformly proper rectangular subtableau of $T.$ Let $T_{C_1}$ be the collection of boxes in $S$ to the left of $U$ and $T_{C_2}$ the boxes of $S$ to the right of $U$ so the $S$ is the horizontal concatentation $S = T_{C_1}UT_{C_2}.$ Then, by Lemma \ref{lem: promotion preserves subtableau}, $P(S) = T_{C_1}'U'T_{C_2}'$ where $U'$ is formed by subtracting 1 from each entry of $U$ and $T_{C'}T_{C_2}'\equiv P(T_C).$ Replacing the maximum entry of $T_{C_2}'$ with the maximum entry $n$ of $T$ yields $P(T_C).$ This completes the proof.
\end{proof}

\begin{remark}\label{rem:k-fold-component}
    An equivalent definition of component promotion when $1\in\partial C$ can be formulated by constructing a specific SYT to represent the component $C.$ Suppose $C$ has the first $\ell$ vertices of $\phi(T)$ as endpoints, so $\partial C = \{t_1<t_2<\ldots<\ldots<t_m\}$ with $t_1=1, \ldots, t_\ell=\ell.$ For each $k\leq \ell,$ we may determine the $k-$fold component promotion of $T_{C}$ as follows:
    \begin{enumerate}
        \item Form the SYT $T_C'$ by replacing entry $t_i$ with $i.$
        \item Compute the $k-$fold promotion $P^k(T_C')$ using the algorithm in Definition \ref{definition: promotion}.
        \item Replace entry $i$ in $P^k(T_C')$ with $\rho^k(t_i).$
    \end{enumerate}
\end{remark}

The results of Lemma \ref{lem:component-promotion-tableau} have a natural geometric interpretation. As in Corollary \ref{cor: left translation}, we see that the only component of $\phi(T)$ whose structure is not preserved in promotion is the component $C_1$ with $1\in\partial C_1.$ If $C$ is any other component of $\phi(T),$ we see that $\rho(C)$ is a component of $\phi(P(T)).$ For instance,  consider our running example in Figure~\ref{fig:promotion-vs-rotation}. All three components nested under the arc $(1,18)$ in $\phi(T)$ look exactly the same in $\rho(\phi(T))$ after rotation; however, the component containing arc $(1,18)$ in $\phi(T)$ looks different both from the rightmost component in $\rho(\phi(T))$ and from the rightmost component of $\phi(P(T))$.

Two conclusions follow:
\begin{enumerate}
    \item Not every component $C$ of $\phi(T)$ gives rise to a component  $\rho(C)$ of $\phi(P(T)).$ 
    \item Nonetheless, if $C$ is the component of $\phi(T)$ containing boundary point $1$, there is a component of $\phi(P(T))$ with exactly the same endpoints as $\rho(C)$.  Moreover, this component is $\phi(P(T_C))$ with $P(T_C)$ as defined in Definition \ref{def:component-promotion}.
\end{enumerate}.  
The next corollary records these observations.

\begin{corollary}\label{cor:endpoints are rotations}
    Choose a sub-diagram $C$ of $\phi(T).$ Suppose that $\partial C = \{a_{i_k}, b_{i_k}\}_{k=1}^s$ is the set of endpoints of $C$ and $T_C$ is the tableau corresponding to $C.$
    \begin{enumerate}
        \item The arcs with endpoints $\rho(\partial C)$ in $\varphi(P(T))$ form a sub-diagram of $\varphi(P(T)).$ 
        \item If $C$ is a component, the arcs with endpoints $\rho(\partial C)$ in $\phi(P(T))$ form the component $\phi(P(T_C)).$
    \end{enumerate}
\end{corollary}

\begin{proof}
    This is a simple application of Lemma \ref{lem:component-promotion-tableau}.
\end{proof}

From Corollary \ref{cor:endpoints are rotations}, we see that the partition of the boundary of $\phi(P(T))$ induced by its components is given by a left-translation of the partition of the boundary of $\phi(T).$ It follows that performing $|\cO(T)|$ left-translations maps the boundary of $\phi(T)$ to itself. We make this precise using a notion of \textit{rotational symmetry}.

\begin{definition}[Rotational Symmetry]
    Let $C_1, \ldots, C_k$ be the components of the $m-$diagram $\phi(T)$ and $\partial C_i$ the set of endpoints of component $C_i.$ Then $\phi(T)$ has order $N$ rotational symmetry if $N$ is the smallest positive integer such that the following equality holds: $$\{\rho^N(\partial C_1), \ldots, \rho^N(C_k)\} = \{\partial C_1, \ldots, \partial C_k\}.$$
\end{definition}

We now show that the orbit length $|\cO(T)|$ is a multiple of the rotational symmetry of its $m-$diagram.

\begin{lemma}\label{lem:rotational symmetry divides orbit}
    Suppose $\phi(T)$ has order $N$ rotational symmetry. Then $N$ divides the orbit length $|\cO(T)|.$
\end{lemma}
\begin{proof}
    First, let $C_1,\ldots, C_k$ be the components of $\phi(T)$ and $C'_1, \ldots, C'_\ell$ the components of $\phi(P(T)).$ For any component $C$, let $\partial C$ be the set of endpoints of $C$. Then by construction, we see that $\sqcup_{i=1}^k \partial C_i$ and $\sqcup_{i=1}^{\ell} \partial C'_i$ are both partitions of the set $[n]=\{1,\ldots, n\}$. Furthermore, by Corollary \ref{cor:endpoints are rotations}, we have that $\bigsqcup_{i=1}^k \rho(\partial C_i)=\bigsqcup_{i=1}^{\ell} \partial C'_i.$ In other words, the partition coming from $\phi(P(T))$  is a one-unit leftward shift modulo $n$ of the partition coming from $\varphi(T)$. In particular, $k=\ell$. 
    
    Since $\varphi(P^{|\cO(T)|}(T))=\varphi(T)$, it follows that $\bigsqcup_{i=1}^k \rho^{|\cO(T)|}(\partial C_i)=\bigsqcup_{i=1}^{k} \partial C_i.$ Since $\phi(T)$ has order $N$ rotational symmetry, we know $N$ is the smallest positive leftward shift modulo $n$ that maps the partition coming from $\phi(T)$ onto itself. Now $|\mathcal{O}(T)| = Nq+r$ for some $q,r\in\mathbb{Z}$ with $0\leq r<N$, and this means $\bigsqcup_{i=1}^k \rho^{r}(\partial C_i)=\bigsqcup_{i=1}^{k} \partial C_i.$  Since $N$ is minimal, we must have that $r=0$ and so $N$ divides $|\mathcal{O}(T)|$.
\end{proof}

\begin{remark}
    For a rectangular tableau $T$, we have seen that the components of the $m-$diagram $\varphi(T)$ induce a specific partition $\Lambda_T=\sqcup_{C\in \varphi(T)} \partial C$ of the set $[|T|]$. Now, let $\Lambda$ be some partition of the set $[nm]$ for some $n,m\in\mathbb{N}$, and consider the set $$\mathcal{T}_{\Lambda}=\left\{T\in SYT(n\times m): \Lambda_T \in \{\Lambda, \rho(\Lambda), \rho^2(\Lambda), \rho^3(\Lambda), \rho^4(\Lambda), \rho^5(\Lambda), \ldots\}\right\}$$ 
    where we interpret $\rho^N(\Lambda)$ to mean applying the map $\rho$ to each subset in $\Lambda$. In other words, $\mathcal{T}_\Lambda$ is the set of tableaux that induce partitions which are rotationally equivalent to $\Lambda$.
    
    Then, for each $T\in\mathcal{T}_\Lambda$, we have that $\cO(T)\subseteq \mathcal{T}_\Lambda$ from Corollary \ref{cor:endpoints are rotations}. In this way, we get a lower bound $|\{\textup{nonempty } \mathcal{T}_\Lambda: \Lambda \textup{ is a partition of } [nm]\}|$ on the number of promotion orbits in $SYT(n\times m)$ and an upper bound $|\mathcal{T}_{\Lambda_T}|$ on $|\cO(T)|$.
\end{remark}

We now provide an algorithm for determining the orbit lengths of non-minimal tableaux. To do so, we introduce an integer $\ell$ that counts rotations until \emph{both} the boundaries of each component match (via rotational symmetry) \emph{and} the arcs above ``look the same."

\begin{theorem}\label{algorithm} 
    The orbit length $|\cO(T)|$ can be determined as follows:

    \begin{enumerate}
        \item Let $N$ be the order of rotational symmetry of $\phi(T).$
        \item Define $\ell$ as
        $$\ell = \min\{k\geq 1: \rho^{kN}(\partial C_i) = \partial C_j \implies P^{kN}(T_{C_i}) = T_{C_j} \forall i\}$$
        where $T_{C_i}$ is the tableau corresponding to $C_i.$
    \end{enumerate}
    
    Then, the equality $|\cO(T)|=\ell N$ holds.
\end{theorem}

\begin{proof}
    By Lemma \ref{lem:rotational symmetry divides orbit} we have that if $\phi(T)$ has order $N$ rotational symmetry, then $N\big||\cO(T)|.$ We complete the proof by showing that $\ell$ as defined in step (4) is the smallest positive integer such that $|\cO(T)|\big| \ell N.$

    Note that for each component $C_i$ of $\phi(T)$ we have that $\rho^{\ell N}(\partial C_i) = \partial C_j$ for some $j$ by the definition of $N.$ Iterated applications Corollary \ref{cor:endpoints are rotations} yields that the component of $\phi(P^{\ell N}(T))$ with endpoints equal to $\partial C_j$ is the component $\phi(P^{\ell N}(T_{C_i})$. Define $C_j'$ to be the component of $\phi(P^{\ell N}(T))$ with endpoints equal to $\partial C_j.$ Since $\ell$ is chosen so that whenever $\rho^{\ell N}(\partial C_i)= \partial C_j$ we have that $P^{\ell N}(T_{C_i}) = T_{C_j},$ we see that the algorithm in Lemma \ref{lemma: creating m-diagram from tableau} constructs the arcs in $C_j'$ in the same order that it constructs the arcs in $C_j.$ It follows that $C_j' = C_j$ and so we have that $\phi(P^{\ell N}(T)) = \phi(T).$ It follows that $|\cO(T)|\big|\ell N.$

    Now, if $k< \ell,$ there is some component $C_i$ of $\phi(T)$ such that $\rho^{kN}(\partial C_i) = C_j$ and $P^{kN}(T_{C_i})\neq T_{C_j}.$ It follows that the component of $P^{kN}(T)$ with endpoints equal to $\partial C_j$ is not equal to $C_j$ and so $|\cO(T)|$does not divide $kN.$ This completes the proof.
\end{proof}

\begin{example}
As an example, both of the $m-$diagrams shown below have the same order of rotational symmetry because both have two components, each with $8$ boundary points.    
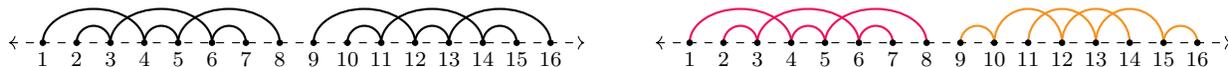
\begin{center}
    \begin{tikzpicture}
        \node at (0,0) {\loneexample};
        \node at (8.6,0) {\ltwoexample};
    \end{tikzpicture}
\end{center}
Thus $N=8$ in the notation of the previous Theorem.  The $m-$diagram on the left has $\ell=1$ while the $m-$diagram on the right has $\ell=2$.  Intuitively, this is because the two components in the $m-$diagram on the left ``look the same" but ``look different" in the $m-$diagram on the right.  Formally, we need to compute promotions of the tableaux corresponding to the four components.  Three of the components above have corresponding tableau that is uniformly equivalent to $T_C$ below, while the fourth is $T_{C'}$ below.
\[T_C = {\ytableausetup{boxsize=0.6cm}\begin{ytableau}
            1 & 2 \\
            3 & 4 \\
            5 & 6 \\
            7 & 8
        \end{ytableau}} \hspace{.5in} T_{C'} = {\ytableausetup{boxsize=0.6cm}\begin{ytableau}
            9 & 11 \\
            10 & 13 \\
            12 & 15 \\
            14 & 16
        \end{ytableau}}\]
As a $4 \times 2$ tableau, note that $T_C$ has promotion order $2$ so $\ell=1$ for the $m-$diagram on the left.  Also $P(T_C)$ is uniformly equivalent to $T_{C'}$.  So $P^8(T_{C'}) \neq T_C$ and $\ell = 2$ for the $m-$diagram on the right, as claimed.
\end{example}

\begin{remark}
    While Theorem \ref{algorithm} gives us a method to compute the orbit length of any non-minimal tableau, for many examples there may be shortcuts. For example, if the $m-$diagram $\phi(T)$ has order $|T|$ rotational symmetry, it is clear that $|\cO(T)| = |T|.$ As well, by the results of \cite{Haiman}, the orbit length $|\cO(T)|$ must divide $|T|.$ It follows that $N$ must divide $|T|,$ reducing the set of possible orders of rotational symmetry to check. Moreover, once $N$ has been found, in step (4) of Theorem \ref{algorithm} we need only check values of $\ell$ such that $\ell N \big||T|.$
\end{remark}

Following the algorithm in Theorem \ref{algorithm}, we see that the orbit length $|\cO(T)|$ is determined by examining orbits of a set of minimal tableaux. While there is no known method for identifying the orbit lengths of minimal tableaux, these tableaux are typically much smaller than $T$ and thus directly computing their orbit lengths is much less computationally expensive than computing the entire promotion orbit of $T.$

\begin{example} We can now calculate the promotion orbit length of our running example $T.$  The $m-$diagram $\phi(T)$ has four components, $C_1, C_2, C_3, C_4.$ 
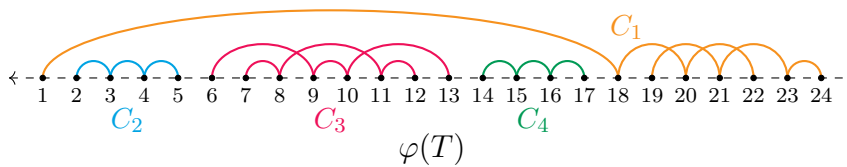
\begin{center}
    \begin{tikzpicture}
        \node at (0,0) {\MTcomponents};
        \node at (0,-1) {$\phi(T)$};
    \end{tikzpicture}
\end{center}

The sets of endpoints for each component are the following:
\begin{align*}
    \partial C_1 &= \{1, 18, 19, 20, 21, 22, 23, 24\}\\
    \partial C_2 &= \{2, 3, 4, 5\} \\
    \partial C_3 &= \{6, 7, 8, 9, 10, 11, 12, 13\} \\
    \partial C_4 &= \{14, 15, 16, 17\}.
\end{align*}
Note that the following equalities hold:
\begin{align*}
    \rho^{12}(\partial C_1) = \partial C_3 \quad
    \rho^{12}(\partial C_2) = \partial C_4 \quad
    \rho^{12}(\partial C_3) = \partial C_1 \quad
    \rho^{12}(\partial C_4) = \partial C_2.
\end{align*}
This is the minimal value of $N$ such that $\sqcup_{i=1}^4\rho^{N}(\partial C_i) = \sqcup_{i=1}^4\partial C_i$ and so $\phi(T)$ has order 12 rotational symmetry. As the orbit length $|\cO(T)|$ must divide 24, either $|\cO(T)|=12$ or $|\cO(T)|=24.$ We now need to calculate the 12-fold promotions of each tableau $T_{C_i}$ corresponding to the components $C_i.$
\begin{center}
    \begin{tikzpicture}
        \node at (0,0) {\ytableausetup{boxsize=0.6cm}\begin{ytableau}
            1 & 19 \\
            18 & 21 \\
            20 & 23 \\
            22 & 24
        \end{ytableau}};
        \node at (0,-1.75) {$T_{C_1}$};
        \node at (2,0) {\begin{ytableau}
            2 \\
            3 \\
            4 \\
            5
        \end{ytableau}};
        \node at (2, -1.75) {$T_{C_2}$};
        \node at (4,0) {\ytableausetup{boxsize=0.6cm}\begin{ytableau}
            6 & 7 \\
            8 & 9 \\
            10 & 11 \\
            12 & 13
        \end{ytableau}};
        \node at (4,-1.75) {$T_{C_3}$};
        \node at (6,0) {\begin{ytableau}
            14 \\
            15 \\
            16 \\
            17
        \end{ytableau}};
        \node at (6, -1.75) {$T_{C_4}$};
    \end{tikzpicture}
\end{center}

Recall that whenever $1\notin \partial C_i,$ the promotion $P(T_{C_i})$ is achieved by simply decrementing each entry of $T_{C_i}$ by 1, which corresponds to simply translating the component $C_i$ once to the left in $\phi(T).$ So, we have that $P^{12}(T_{C_4}) = T_{C_2}.$ We now consider the promotions of $T_{C_1}, T_{C_2}, T_{C_3}.$ As each of these components has an endpoint within the first 12 vertices of $\phi(T),$ the structure of each component will change as the promotion operator is applied to $T.$ First, we see that the 12-fold promotion of $T_{C_1}$ is given by
\begin{center}
\begin{tikzpicture}
    \node at (0,0) {\ytableausetup{boxsize=0.6cm}\begin{ytableau}
            1 & 19 \\
            18 & 21 \\
            20 & 23 \\
            22 & 24
        \end{ytableau}};
        \node at (0,-1.75) {$T_{C_1}$};
    \draw[->](1,0) -- (2, 0);
    \node at (3,0) {\ytableausetup{boxsize=0.6cm}\begin{ytableau}
            17 & 18 \\
            19 & 20 \\
            21 & 22 \\
            23 & 24
        \end{ytableau}};
        \node at (3,-1.75) {$P(T_{C_1})$};
    \draw[->, dotted, thick](4,0) -- (5, 0);
    \node at (6,0) {\ytableausetup{boxsize=0.6cm}\begin{ytableau}
            6 & 7 \\
            8 & 9 \\
            10 & 11 \\
            12 & 13
        \end{ytableau}};
        \node at (6,-1.75) {$P^{12}(T_{C_1})$};
    \end{tikzpicture}
\end{center}
and so $\rho^{12}(\partial C_1)=\partial C_3$ and $P^{12}(T_{C_1}) = T_{C_3}.$ The 12-fold promotion of $T_{C_2}$ is found as follows:
\begin{center}
\begin{tikzpicture}
    \node at (0,0) {\ytableausetup{boxsize=0.6cm}\begin{ytableau}
            2 \\
            3 \\
            4 \\
            5
        \end{ytableau}};
        \node at (0,-1.75) {$T_{C_2}$};
    \draw[->](1,0) -- (2, 0);
    \node at (3,0) {\ytableausetup{boxsize=0.6cm}\begin{ytableau}
            1 \\
            2 \\
            3 \\
            4
        \end{ytableau}};
        \node at (3,-1.75) {$P(T_{C_2})$};
    \draw[->, dotted, thick](4,0) -- (5, 0);
    \node at (6,0) {\ytableausetup{boxsize=0.6cm}\begin{ytableau}
            21 \\
            22 \\
            23 \\
            24
        \end{ytableau}};
        \node at (6,-1.75) {$P^{5}(T_{C_2})$};
        \draw[->, dotted, thick](7,0) -- (8, 0);
    \node at (9,0) {\ytableausetup{boxsize=0.6cm}\begin{ytableau}
            14 \\
            15 \\
            16 \\
            17
        \end{ytableau}};
        \node at (9,-1.75) {$P^{12}(T_{C_2})$};
    \end{tikzpicture}
\end{center}
So we have that $\rho^{12}(\partial C_2)=\partial C_4$ and $P^{12}(T_{C_2})=T_{C_4}.$ Note that we calculate $P^5(T_{C_2})$ from $P(T_{C_2})$ using the arguments in Remark \ref{rem:k-fold-component}. Finally, the 12-fold promotion of $T_{C_3}$ is  
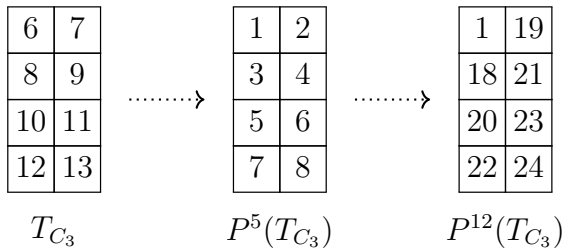
\begin{center}
\begin{tikzpicture}
    \node at (0,0) {\ytableausetup{boxsize=0.6cm}\begin{ytableau}
            6 & 7 \\
            8 & 9 \\
            10 & 11 \\
            12 & 13
        \end{ytableau}};
        \node at (0,-1.75) {$T_{C_3}$};
    \draw[->, dotted, thick](1,0) -- (2, 0);
    \node at (3,0) {\ytableausetup{boxsize=0.6cm}\begin{ytableau}
            1 & 2 \\
            3 & 4 \\
            5 & 6 \\
            7 & 8
        \end{ytableau}};
        \node at (3,-1.75) {$P^5(T_{C_3})$};
    \draw[->, dotted, thick](4,0) -- (5, 0);
    \node at (6,0) {\ytableausetup{boxsize=0.6cm}\begin{ytableau}
            1 & 19 \\
            18 & 21 \\
            20 & 23 \\
            22 & 24
        \end{ytableau}};
        \node at (6,-1.75) {$P^{12}(T_{C_3})$};
    \end{tikzpicture}
\end{center}
and so $\rho^{12}(\partial C_3)=\partial C_1$ and $P^{12}(T_{C_3})=T_{C_1}.$ Thus, the orbit length $|\cO(T)|$ is 12.
\end{example}

\section{Application: Promotion of Rectangular Semi-Standard Tableaux}\label{sec:ssyt}

In this section, we extend the results of Section \ref{section: main} to the case of semi-standard Young tableaux.

We begin with an overview of (column) semi-standard Young tableaux, paying particular attention to the promotion operator. A \emph{(column) semistandard Young tableau (SSYT)} is a filling of the Young diagram for $\lambda$ with the numbers $1, 2, \ldots, n$ so that each number appears at least once, rows strictly increase left-to-right, and columns weakly increase top-to-bottom. If each number $i$ appears $e_i$ times in $T$ then we denote the content $\{1^{e_1}, 2^{e_2}, \ldots, n^{e_n}\}$. Generically, the content of an SSYT is a multiset. 

\begin{remark}
    Usually \emph{semistandard} refers to tableaux in which columns strictly increase and rows are weakly increasing.  For this reason, we stress \emph{column} in our terminology.  Note that each column semistandard tableau is the transpose of a (row) semistandard tableau.
\end{remark}

Given an SSYT of shape $\lambda,$ there is a natural map $\psi$ that constructs an SYT of the same shape. Suppose $T$ has content $\{1^{e_1}, \ldots, n^{e_n}\}.$ The SYT $\psi(T)$ is constructed as follows.
\begin{enumerate}
    \item Relabel each 1 in $T$ sequentially from top to bottom using the set $\{1, 2, \ldots, e_1\}.$
    \item For each $i>1,$ relabel each entry $i$ sequentially from top to bottom using the set $\{(e_1+\ldots+e_{i-1})+1, (e_1+\ldots +e_{i-1})+2, \ldots, (e_1+\ldots+e_{i-1})+e_i\}.$
\end{enumerate}

We give an example of an SYT constructed with $\psi$ in Figure \ref{fig:psi-example}. Later, we will see that the map $\psi$ commutes with the promotion operator in a natural way. First, we prove that $\psi(T)$ is indeed an SYT of the same shape as $T.$

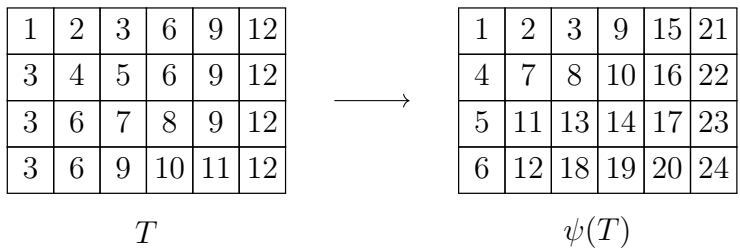
\begin{figure}[h]
    \centering
    \begin{tikzpicture}
        \node at (0,0) {\ytableausetup{boxsize=0.6cm}\begin{ytableau}
            1 & 2 & 3 & 6 & 9 & 12 \\
            3 & 4 & 5 & 6 & 9 & 12 \\
            3 & 6 & 7 & 8 & 9 & 12 \\
            3 & 6 & 9 & 10 & 11 & 12
        \end{ytableau}};
        \node at (0,-1.75) {$T$};
    \draw[->](2.5,0) -- (3.5, 0);
    \node at (6,0) {\ytableausetup{boxsize=0.6cm}\begin{ytableau}
            1 & 2 & 3 & 9 & 15 & 21 \\
            4 & 7 & 8 & 10 & 16 & 22 \\
            5 & 11 & 13 & 14 & 17 & 23 \\
            6 & 12 & 18 & 19 & 20 & 24
        \end{ytableau}};
        \node at (6, -1.75) {$\psi(T)$};
    \end{tikzpicture}
    \caption{The column SSYT $T$ and its corresponding SYT $\psi(T)$}
    \label{fig:psi-example}
\end{figure}

\begin{lemma} \label{lemma: map psi to SYT from SSYT}
    Suppose $T$ is a column SSYT with content $\{1^{e_1}, 2^{e_2}, \ldots, n^{e_n}\}$.  Then $\psi(T)$ is an SYT of the same shape with content $\{1, 2, \ldots, (e_1+\ldots+e_n)\}.$
    \end{lemma}

    \begin{proof}
        As the map $\psi$ simply relabels the entries of $T,$ it is clear that $T$ and $\psi(T)$ are the same shape.

        Since the entries of $T$ strictly increase left-to-right, no number can be repeated twice in the same row. Thus, all copies of each number $i$ in $T$ appear in a skew strip of width at most one.  This means the sequential renumbering for each $i$ is unambiguous. If two copies of $i$ appear in the same column of $T$ then the corresponding boxes in $\psi(T)$ increase strictly from top to bottom.  For all other $i<j$ in $T$, the numbers in the corresponding boxes of $\psi(T)$ are still increasing by construction.  So $\psi(T)$ is both row-strict and column-strict and is thus an SYT.
    \end{proof}

A simple consequence of the definition of $\psi$ is that it is injective on the set of all SSYT with the same shape and content. That is, if $T$ and $S$ are both of shape $\lambda$ and have content $\{1^{e_1}, \ldots, r^{e_r}\},$ then $\psi(T)=\psi(S)$ if and only if $T=S.$

\begin{lemma}\label{lem:psi-injective}
    Suppose $S$ and $T$ are distinct column SSYT of the same shape and have the same content $\{1^{e_1}, \ldots, r^{e_r}\}.$ Then $\psi(T)\neq \psi(S).$
\end{lemma}
\begin{proof}
    Since the contents of $S$ and $T$ are identical, all boxes of $S$ and $T$ with filled with $i$ are relabelled with some integer in the set $\{(e_{1}+\ldots+e_{i-1})+1, \ldots, (e_{1}+\ldots+ e_{i-1})+e_i\}$ by definition of the map $\psi.$ As $S$ and $T$ are distinct column SSYT, there is some pair of integers $j, k$ such that $T_{jk} = i$ and $S_{jk} = i'$ with $i\neq i'.$ Then 
    \[\psi(T)_{jk}\in \{(e_{1}+\ldots+e_{i-1})+1, \ldots, (e_{1}+\ldots+ e_{i-1})+e_i\}\] 
    and 
    \[\psi(S)_{jk}\in \{(e_{1}+\ldots+e_{i'-1})+1, \ldots, (e_{1}+\ldots+ e_{i'-1})+e_i\}.\] 
    These sets are disjoint and so $\psi(T)\neq \psi(S).$
\end{proof}

We now shift our focus towards the promotion of SSYT. Note that the promotion operator from Section \ref{sec: tableaux} generalizes to SSYT by performing the standard promotion once for every 1 that appears in the tableau. 

\begin{definition}[Promotion of SSYT]
    Given a column SSYT $T$ with content $\{1^{e_1}, \ldots, n^{e_n}\}$, the \emph{promotion} of $T$ is the column SSYT $P(T)$ with content $\{1^{e_2}, \ldots, (n-1)^{e_n}, n^{e_1}\}$ created as follows:
    \begin{enumerate}
    \item Erase 1 in the top left corner of $T$ and leave an empty box.
    \item Given the configuration $\ytableausetup {boxsize=0.5 cm} \begin{ytableau} {} & b \\ a \\ \end{ytableau}$ and $b \leq a$ then slide $b$ left; else if $a < b$ slide $a$ up.
    \item Repeat the above process until there are no nonempty boxes below or to the right of the empty box.
    \item Repeat steps (1) through (3) until all $e_1$ copies of $1$ have been removed from the tableau.
    \item Decrement all entries by $1$ and insert $e_1$ copies of $n$ in the empty boxes.
    \end{enumerate}
\end{definition}

Since the promotion of SSYT is in essence a repeated application of promotion for SYT, it is natural to expect that the map $\psi$ from SSYT to SYT and the promotion operator $P$ commute in some natural way. The next result shows that this is true and characterises how the two maps interact.

\begin{lemma} \label{lemma: promoting k-commutes with the map SSYT to SYT}
    Suppose $T$ is a column SSYT with content $\{1^{e_1}, \ldots, r^{e_r}\}$. Then the promotion map commutes with the map $\psi$ to standard Young tableaux in the following sense:
    \[\psi(P(T)) = P^{e_1}(\psi(T)).\]
\end{lemma}
\begin{proof}
    We prove this by induction on $e_1$.  When $e_1=1$ we compare the path of one empty box in the SSYT $T$ and SYT $\psi(T)$ under repeated slides of the form of Step (2) in the promotion algorithm.  The map $\psi$ is defined so that if $a < b$ then the entry of every box in $\psi(T)$ that was labeled $a$ in $T$ is less than the entry of every box in $\psi(T)$ that was labeled $b$ in $T$.  This means that the empty box in Step (2) slides the same direction when promoting $\psi(T)$ as $T$ except perhaps if $a=b$.  In that case, we have the following configurations in $T$ (on the left) and $\psi(T)$ (on the right):
    \[\ytableausetup {boxsize=0.75 cm} \begin{ytableau} {} & a \\ a \\ \end{ytableau} \hspace{1in} {\ytableausetup {boxsize=0.75 cm} \begin{ytableau} {} & {a'} \\ \scalebox{0.7}{$\begin{array}{c} a'\!+\!1 \end{array}$} \\ \end{ytableau}}\]
    For both tableaux, promotion's Step (2) slides the empty box to the right.  Repeating, we conclude that the empty box stays in the same position after each iteration of Step (2).  Thus, two things are true when we first encounter Step (4) in the promotion algorithm: first, the empty box slid along the same path and ended in the same position in $T$ as in $\psi(T)$; and second, when $k=1$ we know $P(\psi(T))=\psi(P(T))$.
    
    Now suppose $k \geq 2$ and take as inductive hypothesis that for any tableau $T'$ with $e_1=k-1$ instances of one, $P^{e_1}(\psi(T'))=P^{k-1}(\psi(T')) = \psi(P(T'))$.  Now let $T$ be a column SSYT for which one occurs $e_1=k$ times in $T$.  Let $T'$ be the column SSYT obtained when we first encounter Step (4) in promoting $T$.  By definition, we know $T'$ is a column SSYT with $k-1$ instances of one and with $|T|-1$ boxes.  Moreover, by the previous paragraph, the tableau obtained at Step (4) of promoting $\psi(T)$ has the same shape as $T'$ and all boxes were slid in exactly the same way.  In other words, the tableau at Step (4) of promoting $\psi(T)$ differs from $\psi(T')$ only by subtracting one from each entry.  By induction, we know that $\psi(P(T'))=P^{k-1}(\psi(T'))$.  By the previous paragraph, we know that $\psi(P(T'))$ differs from $\psi(P(T))$ only by inserting a box with the largest entry of $T$ into the same position as the box with entry $|T|$ in $\psi(P(T))$.  This completes the proof.
\end{proof}

Given a multiset $\{1^{e_1}, \ldots, r^{e_r}\}$, let $R$ be the smallest integer such that $e_{i+R \mod(r)}=e_i$ for all $1\leq i\leq r$. We now provide a formula to compute $|\cO(T)|$ using $|\cO(\psi(T))|$ and $R$.  

\begin{theorem}\label{thm:ssyt-alg}
Let $T$ be a column SSYT with content $\{1^{e_1}, \ldots, r^{e_r}\}$, and let $R$ be as above. Then $|\cO(T)|=\ell R$ where $$\ell\sum_{i=1}^R e_i=\textup{lcm} \left(\sum_{i=1}^R e_i,|\cO(\psi(T))|\right).$$   
\end{theorem}

\begin{proof}
    First observe that because $e_{i+|\cO(T)| \mod(r)}=e_i$ for all $1\leq i\leq r$, it follows by an argument analogous to that of Lemma \ref{lem:rotational symmetry divides orbit} that $R||\cO(T)|$. Say $|\cO(T)|=\ell R$. 
    
    Lemma \ref{lemma: promoting k-commutes with the map SSYT to SYT} says $\psi(P(T))=P^{e_1}(\psi(T)).$ Extending this result gives $\psi(P^{\ell R}(T))=P^{t}(\psi(T))$ where $t=\sum_{i=1}^{\ell R}e_{i \hspace{-.05in}\mod(r)}.$
Observe that we can rewrite $t$ as follows:
\begin{align*}
    t&=\sum_{i=1}^{\ell R}e_{i \hspace{-.05in}\mod(r)}\\
    &=\sum_{i=1}^{R}e_{i \hspace{-.05in}\mod(r)}+\sum_{i=R+1}^{2R}e_{i \hspace{-.05in}\mod(r)}+\ldots + \sum_{i=R(\ell-1)+1}^{\ell R}e_{i \hspace{-.05in}\mod(r)}\\
    &=\ell\sum_{i=1}^{R}e_{i}
\end{align*}
  Since $\psi(P^{\ell R}(T))=\psi(T)=P^t(\psi(T))$, we see that $|\cO(\psi(T))|$ divides $t$. In other words, $t=\ell\sum_{i=1}^{R}e_{i}$ is a multiple of $|\cO(\psi(T))|$. 
  
  For the sake of contradiction, say there is some $0<\ell' <\ell$ such that $\ell' \sum_{i=1}^R e_i$ is a multiple of   $|\cO(\psi(T))|$. Since $\ell' R<\ell R=|\cO(T)|$, this means $P^{\ell ' R}(T)\neq T$ but $P^{t'}(\psi(T))=\psi(T)$ where $t'=\ell'\sum_{i=1}^Re_i$. However, we know by Lemma \ref{lemma: promoting k-commutes with the map SSYT to SYT}, $\psi(P^{\ell ' R}(T))=P^{t'}(\psi(T))$. This means we have two distinct SSYT,  $T$ and $P^{\ell ' R}(T)$, of identical shape and content mapping to the same SYT under $\psi$. By Lemma \ref{lem:psi-injective}, this is not possible and so we conclude that $$\ell\sum_{i=1}^R e_i=\textup{lcm} \left(\sum_{i=1}^R e_i,|\cO(\psi(T))|\right).$$
\end{proof}

\begin{remark}
    The formula in Theorem \ref{thm:ssyt-alg} does not require the assumption that the SSYT $T$ is rectangular. However, if $T$ is not rectangular, then $\psi(T)$ is also not rectangular and thus Theorem \ref{algorithm} cannot be applied to determine the orbit length $|\cO(\psi(T))|.$
\end{remark}

\begin{example}
    We now calculate the orbit length $|\cO(T)|$ for the tableau $T$ in Figure \ref{fig:psi-example}.

    The content of $T$ is $\{1, 2, 3^4, 4, 5, 6^4, 7, 8, 9^4, 10, 11, 12^4\}.$ Note that $e_{i+4 \mod 12} = e_{i}$ for $i=1, \ldots, 12$ and $4$ is the smallest positive integer $R$ such that $e_{i+R\mod 12} = e_{i}$ for all $i.$ Thus, the orbit length $|\cO(T)|$ is a multiple of 4.

    We now need to calculate the orbit length $|\cO(\psi(T))|.$ Since $\psi(T)$ is a rectangular SYT, we use the results of Theorem \ref{algorithm}. First, we construct the $m-$diagram $\phi(\psi(T)).$
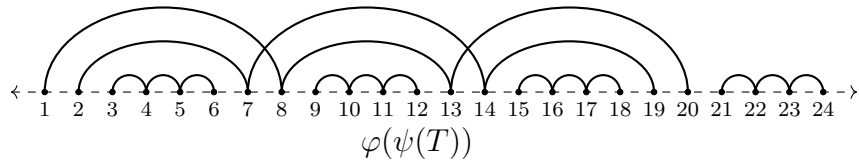
\vspace{1em}
\begin{center}
        \begin{tikzpicture}[scale=0.45]
\draw[style = dashed, <->] (0,0)--(25,0);
\node at (1,-0.5) {\tiny 1};
\node at (2,-0.5) {{\tiny 2}};
\node at (3,-0.5) {{\tiny3}};
\node at (4,-0.5) {{\tiny4}};
\node at (5,-0.5) {{\tiny5}};
\node at (6,-0.5) {\tiny6};
\node at (7,-0.5) {\tiny7};
\node at (8,-0.5) {{\tiny8}};
\node at (9,-0.5) {{\tiny9}};
\node at (10,-0.5) {{\tiny10}};
\node at (11,-0.5) {{\tiny11}};
\node at (12,-0.5) {\tiny12};
\node at (13,-0.5) {\tiny13};
\node at (14,-0.5) {{\tiny14}};
\node at (15,-0.5) {{\tiny15}};
\node at (16,-0.5) {{\tiny16}};
\node at (17,-0.5) {{\tiny17}};
\node at (18,-0.5) {\tiny18};
\node at (19,-0.5) {\tiny19};
\node at (20,-0.5) {{\tiny20}};
\node at (21,-0.5) {{\tiny21}};
\node at (22,-0.5) {{\tiny22}};
\node at (23,-0.5) {{\tiny23}};
\node at (24,-0.5) {{\tiny24}};

\draw[radius=.08, fill=black](1,0)circle;
\draw[radius=.08, fill=black](2,0)circle;
\draw[radius=.08, fill=black](3,0)circle;
\draw[radius=.08, fill=black](4,0)circle;
\draw[radius=.08, fill=black](5,0)circle;
\draw[radius=.08, fill=black](6,0)circle;
\draw[radius=.08, fill=black](7,0)circle;
\draw[radius=.08, fill=black](8,0)circle;
\draw[radius=.08, fill=black](9,0)circle;
\draw[radius=.08, fill=black](10,0)circle;
\draw[radius=.08, fill=black](11,0)circle;
\draw[radius=.08, fill=black](12,0)circle;
\draw[radius=.08, fill=black](13,0)circle;
\draw[radius=.08, fill=black](14,0)circle;
\draw[radius=.08, fill=black](15,0)circle;
\draw[radius=.08, fill=black](16,0)circle;
\draw[radius=.08, fill=black](17,0)circle;
\draw[radius=.08, fill=black](18,0)circle;
\draw[radius=.08, fill=black](19,0)circle;
\draw[radius=.08, fill=black](20,0)circle;
\draw[radius=.08, fill=black](21,0)circle;
\draw[radius=.08, fill=black](22,0)circle;
\draw[radius=.08, fill=black](23,0)circle;
\draw[radius=.08, fill=black](24,0)circle;

\draw[style=thick, color=black] (4,0) arc (0:180: .5cm);
\draw[style=thick, color=black] (5,0) arc (0:180: .5cm);
\draw[style=thick, color=black] (6,0) arc (0:180: .5cm);

\draw[style=thick, color=black] (10,0) arc (0:180: .5cm);
\draw[style=thick, color=black] (11,0) arc (0:180: .5cm);
\draw[style=thick, color=black] (12,0) arc (0:180: .5cm);

\draw[style=thick, color=black] (16,0) arc (0:180: .5cm);
\draw[style=thick, color=black] (17,0) arc (0:180: .5cm);
\draw[style=thick, color=black] (18,0) arc (0:180: .5cm);

\draw[style=thick, color=black] (24,0) arc (0:180: .5cm);
\draw[style=thick, color=black] (23,0) arc (0:180: .5cm);
\draw[style=thick, color=black] (22,0) arc (0:180: .5cm);

\draw[style=thick] (7,0) arc [start angle=0,end angle=180,x radius=2.5cm, y radius=1.5cm];
\draw[style=thick] (13,0) arc [start angle=0,end angle=180,x radius=2.5cm, y radius=1.5cm];
\draw[style=thick] (19,0) arc [start angle=0,end angle=180,x radius=2.5cm, y radius=1.5cm];

\draw[style=thick] (8,0) arc [start angle=0,end angle=180,x radius=3.5cm, y radius=2.5cm];
\draw[style=thick] (14,0) arc [start angle=0,end angle=180,x radius=3.5cm, y radius=2.5cm];
\draw[style=thick] (20,0) arc [start angle=0,end angle=180,x radius=3.5cm, y radius=2.5cm];

\node at (12, -1.5) {$\phi(\psi(T))$};
\end{tikzpicture}
\end{center}

The $m-$diagram $\phi(\psi(T))$ has 5 components $C_1, C_2, C_3, C_4, C_5$ with sets of endpoints
\begin{align*}
    \partial C_1 &= \{1, 2, 7, 8, 13, 14, 19, 20\} \\
    \partial C_2 &= \{3, 4, 5, 6\} \\
    \partial C_3 &= \{9, 10, 11, 12\} \\
    \partial C_4 &= \{15, 16, 17, 18\} \\
    \partial C_5 &= \{21, 22, 23, 24\}. 
\end{align*}
The minimal value of $N$ such that $\sqcup_{i=1}^5\rho^N(\partial C_i) = \sqcup_{i=1}^5\partial C_i$ is $N=6$ and so $\phi(\psi(T))$ has order 6 rotational symmetry. Next, since the orbit length $|\cO(\psi(T))|$ must divide 24, we conclude $|\cO(\psi(T))| = 6, 12,$ or $24.$ We now evaluate the 6-fold promotions of each tableau $T_{C_i}$ corresponding to the components $C_i.$

\begin{center}
    \begin{tikzpicture}
        \node at (0,0) {\ytableausetup{boxsize=0.6cm}\begin{ytableau}
            1 & 2 \\
            7 & 8 \\
            13 & 14 \\
            19 & 20
        \end{ytableau}};
        \node at (0,-1.75) {$T_{C_1}$};
        \node at (2,0) {\begin{ytableau}
            3 \\
            4 \\
            5 \\
            6
        \end{ytableau}};
        \node at (2, -1.75) {$T_{C_2}$};
        \node at (4,0) {\ytableausetup{boxsize=0.6cm}\begin{ytableau}
            9 \\
            10 \\
            11 \\
            12
        \end{ytableau}};
        \node at (4,-1.75) {$T_{C_3}$};
        \node at (6,0) {\begin{ytableau}
            15 \\
            16 \\
            17 \\
            18 \\
        \end{ytableau}};
        \node at (6, -1.75) {$T_{C_4}$};
         \node at (8,0) {\begin{ytableau}
            21 \\
            22 \\
            23 \\
            24 \\
        \end{ytableau}};
        \node at (8, -1.75) {$T_{C_5}$};
    \end{tikzpicture}
\end{center}
Since all entries of $T_{C_3}, T_{C_4}$ and $T_{C_5}$ are greater than 6, we calculate the 6-fold promotion of these tableaux by subtracting 6 from each entry:
\begin{center}
    \begin{tikzpicture}
        \node at (0,0) {\begin{ytableau}
            3 \\
            4 \\
            5 \\
            6
        \end{ytableau}};
        \node at (0, -1.75) {$P^6(T_{C_3})$};
        \node at (2,0) {\ytableausetup{boxsize=0.6cm}\begin{ytableau}
            9 \\
            10 \\
            11 \\
            12
        \end{ytableau}};
        \node at (2,-1.75) {$P^6(T_{C_4})$};
        \node at (4,0) {\begin{ytableau}
            15 \\
            16 \\
            17 \\
            18 \\
        \end{ytableau}};
        \node at (4,-1.75) {$P^6(T_{C_5}).$};
    \end{tikzpicture}
\end{center}
So the equalities $P^6(T_{C_3}) = T_{C_2}, P^6(T_{C_4}) = T_{C_3},$ and $P^6(T_{C_5}) = T_{C_4}$ hold. We now calculate $P^6(T_{C_1})$ and $P^6(T_{C_2})$ using the arguments in Remark \ref{rem:k-fold-component}:
\begin{center}
\begin{tikzpicture}
    \node at (0,0) {\ytableausetup{boxsize=0.6cm}\begin{ytableau}
            1 & 2 \\
            7 & 8 \\
            13 & 14 \\
            19 & 20
        \end{ytableau}};
        \node at (0,-1.75) {$T_{C_1}$};
    \draw[->, dotted, thick](1,0) -- (2, 0);
    \node at (3,0) {\ytableausetup{boxsize=0.6cm}\begin{ytableau}
            5 & 6 \\
            11 & 12 \\
            17 & 28 \\
            23 & 24
        \end{ytableau}};
        \node at (3,-1.75) {$P^2(T_{C_1})$};
    \draw[->, dotted, thick](4,0) -- (5, 0);
    \node at (6,0) {\ytableausetup{boxsize=0.6cm}\begin{ytableau}
            1 & 2 \\
            7 & 8 \\
            13 & 14 \\
            19 & 20
        \end{ytableau}};
        \node at (6,-1.75) {$P^{6}(T_{C_1})$};

    \node at (0,-3.5) {\ytableausetup{boxsize=0.6cm}\begin{ytableau}
            3 \\
            4 \\
            5 \\
            6
        \end{ytableau}};
        \node at (0,-5.25) {$T_{C_1}$};
    \draw[->, dotted, thick](1,-3.5) -- (2, -3.5);
    \node at (3,-3.5) {\ytableausetup{boxsize=0.6cm}\begin{ytableau}
            1 \\
            2 \\
            3 \\
            4 \\
        \end{ytableau}};
        \node at (3,-5.25) {$P^2(T_{C_1})$};
    \draw[->, dotted, thick](4,-3.5) -- (5, -3.5);
    \node at (6,-3.5) {\ytableausetup{boxsize=0.6cm}\begin{ytableau}
            21 \\
            22 \\
            23 \\
            24
        \end{ytableau}};
        \node at (6,-5.25) {$P^{6}(T_{C_1}).$};
    \end{tikzpicture}
\end{center}
So we have that $P^6(T_{C_1})=T_{C_1}$ and $P^6(T_{C_2}) = T_{C_5}$ and thus $\cO(\psi(T)) = 6.$

Using the notation of Theorem \ref{thm:ssyt-alg}, we conclude $R=3$ and $|\cO(\psi(T))| = 6.$ Note that $\sum_{i=1}^R e_i = 6.$ Choosing $\ell$ to satisfy the equality
\begin{align*}
    \ell\sum_{i=1}^Re_i = \ell \cdot 6 = \lcm{6, |\cO(\psi(T))|} = \lcm{6, 6}
\end{align*}
we see that $\ell=1$ suffices. Thus the orbit length $|\cO(T)| = 1\cdot 3 = 3.$

\end{example}

\end{document}